\documentclass[a4paper]{amsart}

\usepackage{amsmath,amsfonts,amssymb,amsthm} %AMS packages
\usepackage[all, cmtip]{xy} %For commutative diagrams
\usepackage{enumerate} %Adds extra options to enumerate environment
\usepackage{pict2e} %Extends picture environment to allow more line slopes, circle radii, etc.
\usepackage{caption} %Gives better caption formatting and more caption formatting options
\usepackage[T1]{fontenc} %Extra font characters

\usepackage{mathtools} %Fixes some problems with amsmath and adds some commands
\mathtoolsset{centercolon} %Corrects colons before equals signs

\usepackage{tikz}
\usepackage{subcaption}

\usepackage{hyperref}

\theoremstyle{definition}
\newtheorem{definition}{Definition}[section]

\theoremstyle{plain}
\newtheorem{lemma}[definition]{Lemma}
\newtheorem{theorem}[definition]{Theorem}
\newtheorem{proposition}[definition]{Proposition}

\theoremstyle{remark}
\newtheorem{remark}[definition]{Remark}

\newcommand{\bA}{\mathbb{A}}

\newcommand{\bC}{\mathbb{C}}

\newcommand{\bH}{\mathbb{H}}

\newcommand{\bN}{\mathbb{N}}

\newcommand{\bP}{\mathbb{P}}
\newcommand{\bQ}{\mathbb{Q}}
\newcommand{\bR}{\mathbb{R}}

\newcommand{\bV}{\mathbb{V}}

\newcommand{\bZ}{\mathbb{Z}}
\newcommand{\WP}{\mathbb{W}\mathbb{P}}

\newcommand{\calA}{\mathcal{A}}

\newcommand{\calD}{\mathcal{D}}

\newcommand{\calL}{\mathcal{L}}
\newcommand{\calM}{\mathcal{M}}
\newcommand{\calN}{\mathcal{N}}
\newcommand{\calO}{\mathcal{O}}

\newcommand{\calS}{\mathcal{S}}
\newcommand{\calT}{\mathcal{T}}

\newcommand{\calX}{\mathcal{X}}
\newcommand{\calY}{\mathcal{Y}}

\DeclareMathOperator{\Gr}{Gr}

\DeclareMathOperator{\NS}{NS}

\DeclareMathOperator{\rank}{rank}

\DeclareMathOperator{\im}{im}

\numberwithin{table}{section}
\numberwithin{figure}{section}

\begin{document}

\title[Threefolds Fibred by Mirror Sextic Double Planes]{Threefolds Fibred by Mirror Sextic Double Planes}

\author[R. Kooistra]{Remkes Kooistra}
\address{The King's University, 9125 -- 50 St NW, Edmonton, AB, T6B 2H3, Canada}
\email{Remkes.Kooistra@kingsu.ca}

\author[A. Thompson]{Alan Thompson}
\address{Department of Mathematical Sciences, Loughborough University, Loughborough, Leicestershire, LE11 3TU, United Kingdom.}
\email{A.M.Thompson@lboro.ac.uk}

\begin{abstract}We present a systematic study of threefolds fibred by K3 surfaces that are mirror to sextic double planes. There are many parallels between this theory and the theory of elliptic surfaces. We show that the geometry of such threefolds is controlled by a pair of invariants, called the generalized functional and generalized homological invariants, and we derive an explicit birational model for them, which we call the Weierstrass form. We then describe how to resolve the singularities of the Weierstrass form to obtain the ``minimal form'', which has mild singularities and is unique up to birational maps in codimension $2$. Finally we describe some of the geometric properties of threefolds in minimal form, including their singular fibres, canonical divisor, and Betti numbers.
\end{abstract}


\date{}
\maketitle

\section{Introduction}

The aim of this paper is to perform a systematic study of threefolds fibred by mirror sextic double planes. A \emph{sextic double plane} is a generic K3 surface of degree two; it is well-known that such K3 surfaces may be realized as double covers of the projective plane ramified along a smooth sextic curve. A \emph{mirror sextic double plane} is a K3 surface that is mirror, in the sense of Dolgachev \cite{mslpk3s} and Nikulin \cite{isbfa}, to a sextic double plane. Mirror sextic double planes are polarized by the rank $19$ lattice 
\[M_1 := H \oplus E_8 \oplus E_8 \oplus A_1,\]
where $E_8$ and $A_1$ are the negative definite root lattices and $H$ is the hyperbolic plane lattice. We henceforth refer to mirror sextic double planes as \emph{$M_1$-polarized K3 surfaces}.

$M_1$-polarized K3 surfaces are a special case of a broader class of K3 surfaces, polarized by the lattices
\[M_n := H \oplus E_8 \oplus E_8 \oplus \langle -2n\rangle,\quad n \in \bN.\]
Such \emph{$M_n$-polarized K3 surfaces} form a natural class to study, as they are mirror to K3 surfaces of degree $2n$. Threefolds fibred by $M_n$-polarized K3 surfaces have been studied in detail by the second author in a series of papers with C. F. Doran, A. Harder, and A. Y. Novoseltsev \cite{flpk3sm,cytfmqk3s,cytfhrlpk3s}. However, these papers almost always assume $n \geq 2$, as  a generic $M_1$-polarized K3 surface admits a polarization-preserving antisymplectic involution which gives rise to novel behaviour. This makes the theory of threefolds fibred by $M_1$-polarized K3 surfaces substantially more complex than the theory of threefolds fibred by $M_n$-polarized K3 surfaces with $n \geq 2$. By performing a detailed study of $M_1$-polarized fibrations, this paper addresses this gap in this program of research.

In addition to addressing the extra complexity presented by the involution, the results of this paper are also significantly stronger than those of its predecessors. In \cite{cytfmqk3s,cytfhrlpk3s}, the authors quickly restrict to studying Calabi-Yau threefolds fibred by $M_n$-polarized K3 surfaces in order to limit the field of inquiry; in particular, the Calabi-Yau assumption imposes strong constraints on the possible values of $n$  \cite[Theorem 2.12]{cytfhrlpk3s}. In this paper we make no such assumption: all of our results hold for arbitrary threefolds fibred by $M_1$-polarized K3 surfaces, although results about $M_1$-polarized Calabi-Yau threefolds, analogous to those in \cite{cytfmqk3s,cytfhrlpk3s}, may be quickly deduced from the main results of this paper.
\medskip

The first main result of this paper is Theorem \ref{thm:Lclassification}, which shows that one-parameter families of lattice polarized K3 surfaces are classified by two invariants. The first is the \emph{generalized functional invariant}, which is a map from the base of the family to an appropriate coarse moduli space, encoding the moduli of the general fibres. The second is the \emph{generalized homological invariant}, a local invariant describing the monodromy around singular fibres. This result is proved in much more generality than required for the rest of the paper, with quite general assumptions on the polarization.

We then restrict ourselves to threefolds fibred by $M_1$-polarized K3 surfaces. In this setting the generalized functional invariant map is a map to the classical modular curve and, once the generalized functional invariant is fixed, there are only two possibilities for the generalized homological invariant around any singular fibre, differing by a choice of sign. Threefolds fibred by $M_1$-polarized K3 surfaces are therefore classified by maps from the base of the fibration to the classical modular curve, along with a choice of sign for each singular fibre.

This classification scheme is the first of many parallels between elliptic surface theory and the theory of threefolds fibred by $M_1$-polarized K3 surfaces. Given these parallels, one might suspect a hidden link between $M_1$-polarized K3 surfaces and elliptic curves; such a link is provided by work of Clingher, Doran, Lewis, and Whitcher \cite{milpk3s,nfk3smmp}. They show that any $M_1$-polarized K3 surface admits a canonically-defined symplectic involution. Quotienting by this involution and resolving singularities, one obtains a Kummer surface $\mathrm{Kum}(E \times E)$ associated to the product of an elliptic curve $E$ with itself. This process canonically associates an elliptic curve $E$ to any $M_1$-polarized K3 surface and this correspondence is bijective; we will discuss it further in Section \ref{sec:genfibres}.

Our next main result is Theorem \ref{thm:weierstrassform}, which shows that threefolds fibred by $M_1$-polarized K3 surfaces admit a birational map to a \emph{Weierstrass form}, which is defined by an explicit equation. This form is very general but, as is also the case for Weierstrass forms of elliptic surfaces, tends to be quite singular. The bulk of this paper is concerned with resolving the singularities of our Weierstrass forms. Theorem \ref{thm:minimalform} shows that, after a series of birational modifications, any threefold fibred by $M_1$-polarized K3 surfaces may be placed in a \emph{minimal form}, which has mild (terminal) singularities and is unique up to birational maps that are isomorphisms in codimension $1$. 

Finally, we compute various properties of the minimal forms. The singular fibres are completely described by Theorem \ref{thm:singularfibres} (see also Table \ref{tab:fibres}) and, under some mild additional assumptions, explicit formulas for the canonical divisor and Betti numbers are given in Sections \ref{sec:canonicalsheaf} and \ref{sec:betti}.

\subsection{Structure of the paper} 

In Section \ref{sec:background} we define lattice polarized families of K3 surfaces and threefolds fibred by lattice polarized K3 surfaces; these definitions form the basis of the remainder of the paper. We then define the generalized functional and homological invariants and prove Theorem \ref{thm:Lclassification}, which shows that these two invariants suffice to classify one-parameter lattice polarized families of K3 surfaces. Finally, we specialize all of this to the $M_1$-polarized case and prove some additional results about the generalized functional and homological invariants.

In Section \ref{sec:weierstrass} we introduce the Weierstrass form of a threefold fibred by $M_1$-polarized K3 surfaces. We begin with a summary of relevant results from \cite{flpk3sm}, which use the Batyrev mirror construction \cite{dpmscyhtv} to construct an explicit toric family of $M_1$-polarized K3 surfaces depending upon two complex parameters $\beta'$ and $\gamma$. This family is described by Proposition \ref{prop:2dimX1}. We then prove Theorem \ref{thm:weierstrassform}, which shows that any threefold fibred by $M_1$-polarized K3 surfaces is birational to a threefold in the form given by Proposition \ref{prop:2dimX1}, where $\beta'$ and $\gamma$ are now meromorphic functions on the base of the fibration. We call this birational model the \emph{Weierstrass form}. In Section \ref{sec:genfibres} we investigate the basic properties of the Weierstrass form: we describe some of the geometry of its general fibres and discuss the link with elliptic surfaces. In Section \ref{sec:canonical} we show (Proposition \ref{prop:canonicalmodel}) that the Weierstrass form of an $M_1$-polarized K3 surface may also be obtained intrinsically from the structure of the $M_1$-polarization, without having to resort to mirror symmetry or toric geometry.

In Section \ref{sec:singularfibres} we embark upon a detailed study of the singular fibres appearing in Weierstrass forms. Proceeding case-by-case, we explicitly describe how to construct birational models for these singular fibres; many of these computations were assisted by the computer algebra systems Sagemath \cite{sage} and Singular \cite{singular}. The results are summarized by Theorem  \ref{thm:singularfibres} and Table \ref{tab:fibres}, which give a complete birational classification of all singular fibres that occur. Again, there are strong parallels with the theory of elliptic surfaces, which inform the naming scheme of our classification. 

In Section \ref{sec:invariants}, we pull together the results of the previous two sections to prove Theorem \ref{thm:minimalform} which shows that, after a series of birational modifications, any threefold fibred by $M_1$-polarized K3 surfaces may be placed in \emph{minimal form}. This form is unique up to birational modifications that are isomorphisms in codimension $1$, has at worst terminal singularities, and has singular fibres as classified by  Theorem \ref{thm:singularfibres}. Finally, we compute some of the properties of the minimal form. Proposition \ref{prop:smoothness} gives a sufficient criterion for the minimal form to be smooth and Proposition \ref{prop:cbf} gives an explicit expression for the canonical divisor. Finally, the results of Section \ref{sec:betti} compute all of the Betti numbers of the minimal form under an extra smoothness assumption and, in the case of the third Betti number, a mild additional assumption on the singular fibre types. These computations use several results from \cite{cytfmqk3s}, which are reproduced in Appendix \ref{appendix} for convenience.

Many of the results from Sections \ref{sec:singularfibres} and \ref{sec:invariants} have analogues for threefolds fibred by $M_n$-polarized K3 surfaces with $n > 1$. These results may all be found stated in their most general forms in \cite{cytfhrlpk3s}; references to this paper are provided for comparison.
\medskip 

\noindent\textbf{Acknowledgments.} The authors are indebted to Andrey Novoseltsev for giving us access to a list of torically induced $M_1$-polarized K3 fibrations on toric Calabi-Yau threefolds that he generated using Sagemath \cite{sage}. This data provided a valuable sanity-check for many of the results contained herein. The second author would also like to thank Charles Doran for numerous useful conversations.

\section{Background} \label{sec:background}

We broadly follow the notational conventions of the related paper \cite{cytfhrlpk3s}.  Let $H$ denote the hyperbolic plane lattice, $E_8$ the negative definite $E_8$ root lattice, and $L$ a non-degenerate primitive sublattice of the K3 lattice $\Lambda_{\mathrm{K3}} := H^{\oplus 3} \oplus E_8^{\oplus 2}$ with signature $(1,\rho-1)$. Let $L^{\perp}$ denote the orthogonal complement of $L$ in $\Lambda_{\mathrm{K3}}$. We begin by defining a notion of $L$-polarization for smooth families of K3 surfaces.

\begin{definition}\label{def:L-pol} \textup{\cite[Definition 2.1]{flpk3sm}}
Let $L \subseteq \Lambda_{\mathrm{K3}}$ be a lattice and $\pi^U\colon \mathcal{X}^U \rightarrow U$ be a smooth family of K3 surfaces over a $1$-dimensional complex manifold $U$. We say that $\mathcal{X}^U$ is an \emph{$L$-polarized family of K3 surfaces} if
\begin{itemize}
\item there is a trivial local subsystem $\mathcal{L}$ of $R^2 \pi_*\mathbb{Z}$ so that, for each $p \in U$, the fibre $\mathcal{L}_p \subseteq H^2({X}_p,\mathbb{Z})$ of $\calL$ over $p$ is a primitive sublattice of $\NS(X_p)$ that is isomorphic to $L$, and
\item there is a line bundle $\calA$ on $\calX^U$ whose restriction $\calA_p$ to any fibre $X_p$ is ample with first Chern class $c_1(\calA_p)$ contained in $\calL_p$ and primitive in $\NS(X_p)$.
\end{itemize}
\end{definition}

Next we extend this definition to threefolds fibred by K3 surfaces, in a way that allows our threefolds to contain singular fibres and mild singularities. Such threefolds will form the principal objects of study in this paper.

\begin{definition} \label{def:LpolK3} We say that $\pi\colon \calX \to B$ is a  \emph{threefold fibred by $L$-polarized K3 surfaces} if all of the following hold.
\begin{itemize}
\item $\calX$ is a threefold with at worst terminal singularities. By \cite[Corollaries 5.38 and 5.39]{bgav}, it follows that any singularities of $\calX$ are isolated points.
\item $B$ is a $1$-dimensional complex manifold. We usually think of $B$ as being either a compact Riemann surface or the open unit disc in $\bC$.
\item $\pi\colon \calX \to B$ is a proper, flat, surjective morphism whose general fibre is a K3 surface.
\item There exists a point $p \in B$, such that the fibre $X_p$ of $\pi$ over $p$ has N\'{e}ron-Severi group $\NS(X_p) \cong L$.
\item Let $U \subset B$ denote the open set over which $\calX$ is smooth and the fibres of $\pi$ are smooth K3 surfaces, and let $\pi^U\colon \calX^U \to U$ denote the restriction of $\calX$ to $U$. Then $\pi^U\colon \calX^U \to U$ is an $L$-polarized family of K3 surfaces, in the sense of Definition \ref{def:L-pol}.
\end{itemize}
\end{definition}

\subsection{The generalized functional and homological invariants}

Let $\pi^U\colon \calX^U \to U$ denote a non-isotrivial $L$-polarized family of K3 surfaces. Assume that there exists a point $p \in U$ such that the N\'{e}ron-Severi group of the fibre $X_p$ over $p$ is isomorphic to $L$. We may associate two invariants to $\calX^U$.

\begin{definition} The \emph{generalized functional invariant} is the map $g\colon U \to \calM_{L}$ from $U$ to the coarse moduli space of $L$-polarized K3 surfaces, which sends $p \in U$ to the point in moduli corresponding to the fibre $X_p$ (c.f. \cite[Remark 3.4]{mslpk3s}).
\end{definition}

Fix a point $p \in U $ such that $\NS(X_p) \cong L$. Monodromy around a loop in $\pi_1(U,p)$ induces an automorphism of $H^2(X_p,\bZ)$, and Definition \ref{def:L-pol} implies that this automorphism must fix $\NS(X_p) \cong L$. We thus obtain an automorphism of the transcendental lattice  $T(X_p) \cong L^{\perp}$.

\begin{definition} The \emph{generalized homological invariant} is the monodromy representation $\rho\colon \pi_1(U,p) \to \mathrm{O}(L^{\perp})$.
\end{definition}

Based on \cite[Theorem 2.3]{cytfmqk3s}, we can prove that these two invariants are sufficient to determine $\calX^U$.

\begin{theorem}\label{thm:Lclassification} Let $\pi^U\colon \calX^U \to U$ denote a non-isotrivial $L$-polarized family of K3 surfaces over a $1$-dimensional complex manifold $U$. Assume that there exists a point $p \in U$ such that the N\'{e}ron-Severi group of the fibre over $p$ is isomorphic to $L$.  Then $\pi^U\colon \calX^U \to U$ is uniquely determined \textup{(}up to isomorphism\textup{)} by its generalized functional and homological invariants.
\end{theorem}
\begin{proof}  For a contradiction, suppose that $\calX^U$ and $\calY^U$ are two non-isomorphic, non-isotrivial $L$-polarized families of K3 surfaces over the same base, with the same generalized functional and homological invariants. On any simply connected open subset $V$ of $U$, Ehresmann's Theorem (see, for example, \cite[Section 9.1.1]{htcagI}) shows that we can choose markings compatible with the $L$-polarizations on the restricted families of K3 surfaces $\calX^U|_V$ and $\calY^U|_V$. Since $\calX^U$ and $\calY^U$ have the same generalized functional invariants, it thus follows from the Global Torelli Theorem \cite[Theorem 2.7']{fagkk3s} that the families $\calX^U|_V$ and $\calY^U|_V$ are isomorphic. 

Thus the families $\calX^U$ and $\calY^U$ must differ in how such open sets are glued together. To construct this gluing data, choose a simply connected open set $V \subset U$ such that $p \in V$. Let $\gamma \subset U$ be a loop, starting and ending at $p$, and let $\{V_i\}_{i=1}^n$ be a cover of $\gamma$ by simply connected open sets in $U$, chosen such that $V_1 = V_n = V$. Then choose fibrewise isomorphisms $\phi_i\colon \calX^U|_{V_i} \to \calY^U|_{V_i}$ for each $i$, so that the gluing maps $\phi_i\circ \phi_{i+1}^{-1}$ on the intersections $V_i \cap V_{i+1}$ are all trivial. This gives rise to a pair of fibrewise isomorphisms $\phi_1,\phi_n \colon \calX^U|_{V} \to \calY^U|_{V}$. The families $\calX^U$ and $\calY^U$ are isomorphic on the union $\cup_{i=1}^n V_i$ if and only if $\psi := \phi_1 \circ \phi_n^{-1}$ is the trivial automorphism on $\calX^U|_{V}$.

Therefore, by assumption, there must exist a loop $\gamma \in \pi_1(U,p)$ so that the fibrewise automorphism $\psi\colon \calX^U|_{V} \to \calX^U|_{V}$ defined by the above construction is non-trivial. In particular the action of $\psi$ on the fibre $X_p$ must be nontrivial. By the Global Torelli Theorem, $\psi$ induces a nontrivial automorphism $\psi^* \in \mathrm{O}(H^2(X_p,\bZ))$ and, by Definition \ref{def:L-pol}, this automorphism must fix $\NS(X_p) \cong L$. Therefore, $\psi^*$ must act nontrivially on $T(X_p) \cong L^{\perp}$.

Let $\rho_{\calX}$ and $\rho_{\calY}$ denote the generalized homological invariants of $\calX^U$ and $\calY^U$. By construction, we have $\rho_{\calY}(\gamma) = \psi^* \rho_{\calX}(\gamma)$. But this contradicts the assumption that $\rho_{\calX} = \rho_{\calY}$.
\end{proof}

Let $\pi\colon \calX \to B$ be a threefold fibred non-isotrivially by $L$-polarized K3 surfaces. By restricting to the open subset $\pi^U\colon \calX^U \to U$ where $\calX$ is a smooth family, we may associate a generalized functional invariant $g\colon U \to \calM_{L}$ and generalized homological invariant $\rho\colon \pi_1(U,p) \to \mathrm{O}(L^{\perp})$ to $\pi\colon \calX \to B$. It follows from Theorem \ref{thm:Lclassification} that $g$ and $\rho$ suffice to uniquely determine $\pi\colon \calX \to B$ up to birational maps affecting only the fibres over $B \setminus U$.

\subsection{Threefolds fibred by \texorpdfstring{$M_1$}{M1}-polarized K3 surfaces}

Let $H$ denote the hyperbolic plane lattice and let $A_1$ and $E_8$ denote the negative definite $A_1$ and $E_8$ root lattices, respectively. Let $M_1$ denote the rank $19$ lattice $M_1 := H \oplus E_8 \oplus E_8 \oplus A_1$. Throughout the remainder of this paper, we will investigate the case $L = M_1$. Note that the orthogonal complement of $M_1$ in $\Lambda_{\mathrm{K3}}$ is $M_1^{\perp} = H \oplus \langle 2 \rangle$.

In this setting, the assumption that $\calX^U$ is not an isotrivial family, combined with the fact that $\calM_{M_1}$ has dimension $20 - \rank(M_1) = 1$, implies that the image of the generalized functional invariant map $g$ is open in $\calM_{M_1}$. Thus, at a general point $p \in U$, the N\'{e}ron-Severi group of the fibre over $p$ is isomorphic to $M_1$.

In this setting we can make some stronger statements about the generalized functional and homological invariants. Firstly, a result of Dolgachev \cite[Theorem 7.1]{mslpk3s} shows that the coarse moduli space $\calM_{M_1}$ of $M_1$-polarized K3 surfaces is isomorphic to the classical modular curve $\bH/\mathrm{PSL}(2,\bZ)$. This modular curve has a canonical compactification, usually denoted $X(1)$, which we denote by $\overline{\calM}_{M_1}$; it is a genus $0$ orbifold curve with orbifold points of orders $(2,3,\infty)$.

\begin{definition} Let $\pi\colon \calX \to B$ be a threefold fibred non-isotrivially by $M_1$-polarized K3 surfaces. The generalized functional invariant map $g\colon U \to \calM_{M_1}$ is a finite morphism between curves, so extends uniquely to a map $\overline{g}\colon B \to \overline{\calM}_{M_1}$. We call $\overline{g}$ the \emph{extended generalized functional invariant}.
\end{definition}

Using the ideas from the proofs of Theorem \ref{thm:Lclassification} and \cite[Theorem 2.3]{cytfmqk3s}, we can also prove a stronger result about the generalized homological invariant in this setting.

\begin{proposition} \label{prop:plusminus} Let $\calX^U \to U$ and $\calY^U \to U$ denote two non-isotrivial $M_1$-polarized families of K3 surfaces over a $1$-dimensional complex manifold $U$, with the same generalized functional invariants, and let $\rho_{\calX}$ and $\rho_{\calY}$ denote their generalized homological invariants. Then for general choice of $p \in U$ and any loop $\gamma \in \pi_1(p,U)$, we have $\rho_{\calX}(\gamma) = \pm \rho_{\calY}(\gamma)$.
\end{proposition}
\begin{proof} Since the generalized homological invariant is defined uniquely by the fibration, if there is a loop $\gamma \in \pi_1(U,p)$ with $\rho_{\calX}(\gamma) \neq \rho_{\calY}(\gamma)$, then $\calX^U$ and $\calY^U$ are not isomorphic. By the proof of Theorem \ref{thm:Lclassification}, there thus exists an nontrivial automorphism $\psi\colon X_p \to X_p$ such that  $\rho_{\calY}(\gamma) = \psi^* \rho_{\calX}(\gamma)$.

Recall that $\psi^*$ must fix $\NS(X_p) \cong M_1$ and act nontrivially on $T(X_p) \cong M_1^{\perp}$; $\psi$ is thus a non-symplectic automorphism of $X_p$. Define $\mathrm{O}(M_1^\perp)^* \subset \mathrm{O}(M_1^\perp)$ to be the kernel of the natural map  $\mathrm{O}(M_1^\perp) \to A_{M_1^{\perp}}$, where $A_{M_1^{\perp}}$ denotes the discriminant group. Then by \cite[Proposition 1.6.1]{isbfa}, $\psi^*$ descends to an element of $\mathrm{O}(M_1^\perp)^*$. Furthermore, since $X_p$ is general, $M_1^\perp$ supports an irreducible rational Hodge structure, so $\psi^*$ must act irreducibly on $M_1^\perp$. Therefore, by \cite[Theorem 3.1]{fagkk3s}, it follows that the order $n$ of $\psi^*$ must satisfy $\varphi(n) | \rank(M_1^\perp) = 3$, where $\varphi(n)$ denotes Euler's totient function. Using Vaidya's \cite{ietf} lower bound  for $\varphi(n)$,
\[\varphi(n) \geq \sqrt{n}\  \mathrm{for}\  n > 2,\  n \neq 6\]
we see that $\varphi(n) |3$ implies that $n \leq 9$. By checking all the possibilities, we see that $n = 2$ or $n=1$. We cannot have $n=1$, since $\psi^*$ is non-trivial. So $\psi^*$ must have order $2$. As $\psi^*$ acts irreducibly on $M_1^{\perp}$ it has no nontrivial invariant subspace; the only option is $\psi^* = -\mathrm{Id}$. 
\end{proof}

\begin{remark}\label{rem:antisymplectic} The sign indeterminacy in this proposition corresponds to the existence of an antisymplectic involution on a general $M_1$-polarized K3 surface that preserves the $M_1$-polarization. This involution will be described explicitly in Section \ref{sec:genfibres}.
\end{remark}

\section{Weierstrass form} \label{sec:weierstrass}

\label{sec:normalforms} In this section we follow \cite[Section 5.4]{flpk3sm} to derive a normal form for $M_1$-polarized K3 surfaces, then exhibit some of the properties of this normal form. 

An $M_1$-polarized K3 surface is mirror to a $\langle 2 \rangle$-polarized K3 surface, which can generically be expressed as a hypersurface of degree $6$ in $\WP(1,1,1,3)$. By the Batyrev mirror construction \cite{dpmscyhtv}, a generic $M_1$-polarized K3 surface can therefore be realized torically as an anticanonical hypersurface in $\WP(1,1,1,3)^{\circ}$, the polar dual of $\WP(1,1,1,3)$. 

We derive a normal form for a generic such anticanonical hypersurface following the method of \cite[Section 4.2]{msag}. Let $M := \bZ^3$ denote the standard lattice and let $M_{\bR} := M \otimes \bR$. An anticanonical hypersurface in $\WP(1,1,1,3)$ corresponds to the reflexive polytope $\Delta \subset M_{\bR}$ with vertices
\begin{align*} u_0 &:= (-1,-1,-1) \\
u_1 &:= (5,-1,-1) \\
u_2 &:= (-1,5,-1) \\
u_3 &:= (-1,-1,1).
\end{align*}
The toric variety $\WP(1,1,1,3)^{\circ}$ is determined by the fan $\Sigma^{\circ} \subset M_{\bR}$ whose rays are generated by the above set of vectors. These vectors generate a sublattice $M' \subset M$ of index $12$. The quotient $M/M'$ is the group
\[G:=\{(d_0,d_1,d_2,d_3) \in \bZ_6^3 \oplus \bZ_2 : d_0+d_1+d_2 + 3d_3 \equiv 0 \ (\text{mod}\ 6)\}/\bZ_6,\]
where $\bZ_6 \subset \bZ_6^3 \oplus \bZ_2$ is embedded diagonally and $\bZ_n := \bZ/n\bZ$ denotes integers modulo $n$.

The vectors $u_1,u_2,u_3$ form a basis for $M'$ and $u_0 = -u_1-u_2-3u_3$, so if we view $\Sigma^{\circ}$ as a fan in $M' \otimes \bR$, then $\Sigma^{\circ}$ is just the standard fan for $\WP(1,1,1,3)$. Switching to the overlattice $M$, there is a $G$-action on $\WP(1,1,1,3)$,
\[(x_0,x_1,x_2,x_3) \longmapsto (\mu^{d_0}x_0, \mu^{d_1}x_1,\mu^{d_2}x_2,\mu^{3d_3}x_3),\]
where $(d_0,d_1,d_2,d_3) \in G$ and $\mu$ is a primitive sixth root of unity.  $\WP(1,1,1,3)^{\circ}$ is the quotient of $\WP(1,1,1,3)$ by this $G$-action.

Next we need to obtain the equation of a generic anticanonical hypersurface in $\WP(1,1,1,3)^{\circ}$. The homogeneous coordinate ring of $\WP(1,1,1,3)^{\circ}$ is given by $\bC[x_0,x_1,x_2,x_3]$ graded by the Chow group $A_2(\WP(1,1,1,3)^{\circ})$. We have an exact sequence
\[0 \longrightarrow \bZ^3 \longrightarrow \bZ^4 \stackrel{\delta}{\longrightarrow} \bZ \oplus G \longrightarrow 0,\]
where the first map is given by the matrix whose rows are the vectors $u_i$, and the map $\delta$ is given by 
\[\delta(d_0,d_1,d_2,d_3) := (d_0+d_1+d_2 + 3d_3, (-d_1-d_2-3d_3, d_1,d_2,d_3)).\]
Thus $A_2(\WP(1,1,1,3)^{\circ}) \cong \bZ \oplus G$ and the grading on $\bC[x_0,x_1,x_2,x_3]$ is given by taking $x_0^{d_0}x_1^{d_1}x_2^{d_2}x_3^{d_3}$ to $\delta(d_0,d_1,d_2,d_3)$.

The anticanonical class of $\WP(1,1,1,3)^{\circ}$ has degree $\delta(1,1,1,1) = (6,0)$. A generic polynomial in those monomials in $\bC[x_0,x_1,x_2,x_3]$ with degree $(6,0)$ is
\[a_0x_0^6 + a_1x_1^6 + a_2x_2^6 + a_3x_3^2 + a_4x_0x_1x_2x_3 + a_5x_0^2x_1^2x_2^2.\]
Using the torus action, we rescale the variables in this equation by
\[ x_0 \mapsto a_1^{\frac{1}{6}}a_2^{\frac{1}{6}}a_3^{\frac{1}{2}}a_4^{-1}x_0,\quad x_1 \mapsto a_1^{-\frac{1}{6}}x_1, \quad x_2 \mapsto a_2^{-\frac{1}{6}} x_2, \quad x_3 \mapsto a_3^{-\frac{1}{2}}x_3,\]
to obtain
\begin{equation}\label{eq:toricform}\alpha x_0^6 + x_1^6 + x_2^6 + x_3^2 + x_0x_1x_2x_3 + \beta x_0^2x_1^2x_2^2,\end{equation}
for $\alpha = \frac{a_0a_1a_2a_3^3}{a_4^6}$ and $\beta = \frac{a_3a_5}{a_4^2}$. This is a homogeneous polynomial in $\WP(1,1,1,3)$ with variables $x_0,x_1,x_2$ of weight $1$ and $x_3$ of weight $3$, which defines a hypersurface; an anticanonical hypersurface in $\WP(1,1,1,3)^{\circ}$ is given by the quotient of such a hypersurface by the group $G$.

The quotient $\WP(1,1,1,3)/G$ is realized by the map
\begin{align*}
\psi\colon \WP(1,1,1,3) &\longrightarrow \bP^5 \\
(x_0,x_1,x_2,x_3) &\longmapsto (x_0^6,x_1^6,x_2^6,x_3^2,x_0x_1x_2x_3,x_0^2x_1^2x_2^2),
\end{align*}
so $\WP(1,1,1,3)^{\circ}$ is isomorphic to the image of this map. This image is defined by the equations
\[ y_0y_1y_2 - y_5^3 = y_3y_5-y_4^2 = 0,\]
where $(y_0,\ldots,y_5)$ are homogeneous coordinates on $\bP^5$. The map $\psi$ takes the family of anticanonical hypersurfaces defined by Equation (1) to the hyperplane
\[\alpha y_0 + y_1 + y_2 + y_3 + y_4 + \beta y_5 = 0.\] 
We thus obtain a family of surfaces
\begin{equation} \label{eq:compactWeierstrass} \{\alpha y_0 + y_1 + y_2 + y_3 + y_4 + \beta y_5 =y_0y_1y_2 - y_5^3 = y_3y_5-y_4^2 = 0 \} \subset \bP^5.\end{equation}

Working on the maximal torus $y_4 = 1$, $(y_0,y_1,y_2,y_3,y_5) \in (\bC^*)^5$, we may solve for $y_0 = \frac{1}{y_1y_2y_3^3}$ and $y_5 = \frac{1}{y_3}$ in these expressions. On this maximal torus, the family defined by Equation \eqref{eq:compactWeierstrass} is isomorphic to the vanishing locus in $(\bC^*)^3$ of the rational polynomial
\begin{equation} \label{eq:openWeierstrass} x+ y + z + \frac{\alpha}{x^3yz} + \frac{\beta}{x} + 1 = 0,\end{equation}
where we have set $x := y_3$, $y := y_2$, and $z := y_1$ for clarity. This is the open form used in \cite[Section 5.4]{flpk3sm}.

We introduce a new parameter
\[\gamma := \frac{1728 \alpha}{(4\beta-1)^3}.\]
From the results of \cite[Section 5.4]{flpk3sm}, $\gamma$ parametrizes the compact moduli space $\overline{\calM}_{M_1}$. We find that $\gamma = \infty$ at the orbifold point of order $3$, $\gamma = -1$ at the orbifold point of order $2$, and $\gamma = 0$ at the infinite order (cusp) point.

\begin{remark}\label{rem:jgamma} The usual parameter $j$ on $\overline{\calM}_{M_1} \cong X(1)$ is related to $\gamma$ by  $j = -\frac{1}{\gamma}$. For a smooth $M_1$-polarized K3 surface with generalized functional invariant $\gamma$, this $j$ can be interpreted as the $j$-invariant of a canonically associated elliptic curve; see Section \ref{sec:genfibres}.
\end{remark}

By \cite[Lemma 5.17]{flpk3sm}, after performing a minimal resolution of singularities, Equation \eqref{eq:compactWeierstrass} defines  a family of $M_1$-polarized K3 surfaces over the base $\{(\alpha,\beta) \in \bC^2 : \gamma = \frac{1728 \alpha}{(4\beta-1)^3} \notin \{0,-1,\infty\}\}$. After changing variables 
\[y \mapsto \frac{(4\beta - 1)y}{4x},\quad z \mapsto \frac{(4\beta - 1)z}{4x}\] 
and completing the square in $x$, we may re-write \eqref{eq:openWeierstrass} as
\begin{equation} \label{eq:2dimopenX1} \frac{x^2}{4\beta -1} + y  + z + \frac{\gamma}{27yz} + 1 = 0. \end{equation}
This reparametrizes our family over the base $\{(\beta,\gamma) \in \bC^2 : \beta \neq \frac{1}{4},\ \gamma \notin \{0,-1\}\}$. For simplicity of notation, we introduce a new parameter $\beta' := \frac{1}{4\beta - 1}$.

\begin{proposition} \label{prop:2dimX1} The family given by Equation \eqref{eq:2dimopenX1} admits a compactification to a family of singular quartic K3 surfaces in $\bP^3[w,x,y,z]$ over the base $\{(\beta',\gamma) \in \bC^2 : \beta' \neq 0,\ \gamma \notin \{0,-1\}\}$, given by
\begin{equation}
\label{eq:2dimX1}\frac{\gamma}{27} w^4 + w^2yz + wy^2z + wyz^2 + \beta' x^2yz = 0.
\end{equation}
This family admits a minimal resolution to a $2$-parameter $M_1$-polarized family of K3 surfaces with generalized functional invariant given by projection onto the $\gamma$ variable. The fibres of the generalized functional invariant map give rise to families of $M_1$-polarized K3 surfaces that are isotrivial but not trivial: monodromy around $\beta' \in \{0, \infty\}$ acts on the transcendental lattice as $-\mathrm{Id}$.
\end{proposition}
\begin{proof} This result is a restatement of \cite[Lemma 5.17]{flpk3sm} and the first part of \cite[Proposition 5.18]{flpk3sm} in the notation of this paper.
\end{proof}

\begin{remark}
Whilst we have followed the discussion in \cite[Section 5.4]{flpk3sm} above for consistency, we note that the family \eqref{eq:2dimX1} can be obtained much more directly from \eqref{eq:compactWeierstrass} as follows. Working on the affine chart $y_5 = 1$, we may write $y_3 = y_4^2$ and $-y_2 = \alpha y_0 + y_1 + y_3 + y_4 + \beta$, from which we obtain the affine hypersurface
$\{yz(x^2 + x + \alpha y + z + \beta) + 1 = 0\} \subset \bA^3$, where we have set $y:= y_0$, $z:=y_1$, and $x := y_4$. After compactifying to $\bP^3$, completing the square in $x$, and rescaling variables we obtain Equation \eqref{eq:2dimX1}.
\end{remark}

If we set $\beta' \neq 0$ to be a constant in Equation \eqref{eq:2dimopenX1} and rescale $x$, we obtain a special family over $\{\gamma \in \bC : \gamma \notin \{0,-1\}\}$ that may be written on the maximal torus $(\bC^*)^3$ as 
\begin{equation} \label{eq:openX1} x^2 + y + z + \frac{\gamma}{27yz} + 1 = 0.\end{equation}

\begin{proposition} \label{prop:X1} The family given by Equation \eqref{eq:openX1} admits a compactification to a  family $\overline{\calX_1}$ of singular quartic K3 surfaces in $\bP^3[w,x,y,z]$ over $\overline{\calM}_{M_1}$, with parameter $\gamma$, given by 
\begin{equation}
\label{eq:X1} \frac{\gamma}{27} w^4 + w^2yz + wy^2z + wyz^2 + x^2yz = 0.
\end{equation}
Its minimal resolution is a family $\calX_1 \to \overline{\calM}_{M_1}$, whose fibres over $\gamma \notin \{0,-1,\infty\}$ form an $M_1$-polarized family of K3 surfaces. Moreover, the family $\calX_1$ is \emph{modular}, in the sense that its generalized functional invariant map is an isomorphism.
\end{proposition}
\begin{proof} This is a restatement of the second part of \cite[Proposition 5.18]{flpk3sm} in the notation of this paper.\end{proof}

\begin{remark} The family $\calX_1$ has been previously studied by a number of authors, who have used various different birational models according to their requirements and personal tastes. In addition to \cite[Section 5.4]{flpk3sm}, which we have already mentioned, it appears in the context of mirror symmetry in work of Dolgachev \cite[Example 8.3]{mslpk3s} and Przyjalkowski et al. (see the survey paper \cite{tlgm}, where it appears as family $X_{1-1}$, and references therein), and in the context of Picard-Fuchs equations in work of Smith \cite[Example 2.15, case $\calD$]{pfdefk3s} and Doran and Malmendier \cite[Section 6.1, case $n = 1$]{cymrsrmt}. 
\end{remark}

The following lemma computes the periods of the family $\calX_1$.

\begin{lemma} \label{lem:X1monodromy} The local monodromy matrices of the transcendental variation of Hodge structure $\calT(\calX_1)$ associated to $\calX_1$ are conjugate to
\[\Gamma_{-1} := \begin{pmatrix} {-1} & 0 & 0 \\ 0 & {1} & 0 \\ 0 & 0 & 1 \end{pmatrix},\quad 
\Gamma_{\infty} := \begin{pmatrix} -1 & 0 & 0 \\ 0 & \omega & 0 \\ 0 & 0 & \omega^5 \end{pmatrix},\quad
\Gamma_{0} := \begin{pmatrix} 1 & 1 & 0 \\ 0 & 1 & 1 \\ 0 & 0 & 1 \end{pmatrix},\]
around $\gamma = -1$, $\gamma = \infty$, and $\gamma = 0$, respectively, where $\omega$ denotes a primitive sixth root of $1$.
\end{lemma}
\begin{proof} The periods of the family $\calX_1$ were explicitly computed by Doran and Malmendier \cite[Lemma 6.6]{cymrsrmt}, who found that they are given by the hypergeometric function $\mbox{}_3F_2(\frac{1}{6},\frac{1}{2},\frac{5}{6};1,1|j)$ (recall here that $j = -\frac{1}{\gamma}$; see Remark \ref{rem:jgamma}). From this, the global monodromy representation of the transcendental variation of Hodge structure $\calT(\calX_1)$ associated to $\calX_1$ can be computed using a theorem of Levelt \cite[Theorem 1.1]{leveltthesis} (see also \cite[Theorem 3.5]{mhff}). This computation gives the required forms for the local monodromy matrices.\end{proof}

Together with Proposition \ref{prop:2dimX1}, we can now compute the generalized homological invariant of any family given in the form \eqref{eq:2dimX1}. This will be particularly useful when we come to study singular fibres in Section \ref{sec:singularfibres}.

\begin{proposition} \label{prop:generalmonodromy} Let $\beta'(\delta)$ and $\gamma(\delta)$ be meromorphic functions on the open complex unit disc $\Delta$ such that $\beta'(\delta) \notin \{0,\infty\}$ and $\gamma(\delta) \notin \{0,-1,\infty\}$ for all $\delta$ in the punctured disc $\Delta^* = \Delta \setminus \{0\}$. Suppose that $\beta'(0)$ takes a value in $\{0,\infty\}$ with order $b \geq 0$, and $\gamma(0)$ takes a value in $\{0,-1,\infty\}$ with order $d \geq 0$. Then the transcendental monodromy of the $M_1$-polarized family of K3 surfaces defined over $\Delta^*$ by Equation \eqref{eq:2dimX1} is conjugate to $(-\mathrm{Id})^b\Gamma_{\gamma(0)}^d$, where $\Gamma_{\gamma(0)}$ is defined as in Lemma \ref{lem:X1monodromy}.
\end{proposition}
\begin{proof} Proposition \ref{prop:2dimX1} tells us that the $2$-parameter family defined by Equation \eqref{eq:2dimX1} has transcendental monodromy $-\mathrm{Id}$ around $\beta' \in \{0,\infty\}$, and it follows from Lemma \ref{lem:X1monodromy} that its transcendental monodromy around $\gamma \in \{0,-1,\infty\}$ is given by the corresponding matrix $\Gamma_{0}$, $\Gamma_{-1}$, $\Gamma_{\infty}$. The result then follows by pulling-back by the map defined by $(\beta'(\delta),\gamma(\delta))$, noting that $(-\mathrm{Id})$ commutes with $\Gamma_{\gamma(0)}$.
\end{proof}

Now we can show that Equation \eqref{eq:2dimX1} is a kind of ``Weierstrass form'' for threefolds fibred by $M_1$-polarized K3 surfaces.

\begin{theorem} \label{thm:weierstrassform} Let $\pi\colon \calX \to B$ be a threefold fibred non-isotrivially by $M_1$-polarized K3 surfaces. Then $\calX$ is birational over $B$ to a threefold $\overline{\pi}\colon\overline{\calX} \to B$ defined by Equation \eqref{eq:2dimX1} where $\beta'$ and $\gamma$ are  meromorphic functions on $B$.
\end{theorem} 

\begin{proof} Let $\overline{g} \colon B \to \overline{\calM}_{M_1}$ denote the extended generalized functional invariant map of $\calX$. Composing with the parametrization $\gamma$ of $\overline{\calM}_{M_1}$ gives the required map $\gamma\colon B \to \bP^1$.

It remains to define the map $\beta'$. Let $\Sigma \subset B$ denote the closed set over which $\calX$ is either singular, has a singular fibre, or has a smooth fibre with $\gamma \in \{0,-1,\infty\}$, and let $U := B \setminus \Sigma$ denote its open complement. Let $g \colon U \to \calM_{M_1}$ denote the generalized functional invariant and let $\rho_{\calX}$ denote the generalized homological invariant. 

Now let $\calY:=g^*\calX_1$ denote the pull-back of the family $\calX_1 \to {\calM}_{M_1}$ to $U$ and let $\rho_{\calY}$ denote its generalized homological invariant. Note that $\calY$ may be written in the form of Equation \eqref{eq:2dimX1}, with $\gamma$ as above and $\beta' \notin \{0,\infty\}$ a constant. We proceed by comparing $\rho_{\calX}$ with $\rho_{\calY}$.

Let $D \subset \Sigma$ denote the finite set of points $p$ where $\rho_{\calX}(\sigma) \neq \rho_{\calY}(\sigma)$ for a simple loop $\sigma$ around $p$; by Proposition \ref{prop:plusminus}, we must therefore have $\rho_{\calX}(\sigma) = -\rho_{\calY}(\sigma)$ for all such points. The map $\varphi\colon \pi_1(U)\to \{\pm 1\}$ defined by $\sigma \to \det(\rho_{\calX}(\sigma))\det(\rho_{\calY}(\sigma))$ is a permutation representation which takes the value $-1$ if $\sigma$ is a simple loop around a point in $D$ and $1$ for simple loops around other points; consequently, $D$ must contain an even number of points. Therefore, there exists a meromorphic function $\beta'$ on $B$ which has either a simple pole or a simple zero at each point of $\Sigma$, and no other poles or zeros.

Let $\overline{\pi}\colon \overline{\calX} \to B$ denote the family defined by Equation \eqref{eq:2dimX1} with these functions $\beta'$ and $\gamma$, and let $\widetilde{\calX}$ denote its minimal resolution. By construction $\calX$ and $\widetilde{\calX}$ have the same generalized functional invariant and, by Proposition \ref{prop:generalmonodromy}, the same homological invariant. By  Theorem \ref{thm:Lclassification}, $\calX$ and $\overline{\calX}$ are isomorphic over $U$, hence birational over $B$.
\end{proof}

\begin{definition} Let $\pi\colon \calX \to B$ be a threefold fibred non-isotrivially by $M_1$-polarized K3 surfaces. A threefold $\overline{\pi}\colon \overline{\calX} \to B$ defined by Equation \eqref{eq:2dimX1}, with $\beta'$ and $\gamma$ meromorphic functions on $B$, which is birational to $\calX$ over $B$ is called a \emph{Weierstrass form} of $\pi\colon \calX \to B$. In a slight abuse of terminology, we will often call $\gamma$ the \emph{generalized functional invariant} of $\overline{\pi}\colon \overline{\calX} \to B$.
\end{definition}

From Theorem \ref{thm:weierstrassform}, such Weierstrass forms always exist.

\subsection{General fibres}\label{sec:genfibres}

In this subsection, we study the general fibres of the $2$-parameter family \eqref{eq:2dimX1} from Proposition \ref{prop:2dimX1} and show how to resolve their singularities. This allows us to explicitly identify the $M_1$-polarizations on these fibres and study some of their geometry. 

We begin with a quartic surface defined by Equation \eqref{eq:2dimX1} for some $\gamma \notin \{0,-1,\infty\}$ and $\beta' \notin \{0,\infty\}$. All of the singularities in this surface occur on the locus $\{w = 0\}$, which consists of the three lines $L_x := \{ w = x = 0\}$, $L_y := \{w = y = 0\}$, and $L_z := \{w = z = 0\}$. Singularities appear at the four points:
\begin{itemize}
\item $P_1:= \{w = x = y =0\}$, which is an $A_7$ singularity;
\item $P_2:= \{w = x = z = 0\}$, which is an $A_7$ singularity;
\item $P_3:= \{w = y = z = 0\}$, which is an $A_3$ singularity; and
\item $P_4 := \{w = x = y+z = 0\}$, which is an $A_1$ singularity.
\end{itemize}

These are all Du Val singularities, which may be resolved in the usual way. After doing so, we obtain a K3 surface containing a configuration of rational $(-2)$-curves with the dual graph shown in Figure \ref{fig:dualgraph}. In this picture, the $E_i$ (resp. $F_i$) are exceptional curves arising from the resolution of the $A_7$ singularity at $P_1$ (resp. $P_2$), the $D_i$ are exceptional curves arising from the resolution of the $A_3$ singularity at $P_3$, and $C$ is the exceptional curve arising from the resolution of the $A_1$ singularity at $P_4$.

\begin{figure}
\begin{tikzpicture}[scale=0.8,thick]
\draw (1,3)--(1,1);
\draw (0,1)--(14,1);
\draw(13,3)--(13,1);
\draw(1,3)--(7,3)--(13,3);
\draw(7,1)--(7,2);

\node[below] at (0,1) {$E_1$};
\node[below] at (1,1) {$E_2$};
\node[below] at (2,1) {$E_3$};
\node[below] at (3,1) {$E_4$};
\node[below] at (4,1) {$E_5$};
\node[below] at (5,1) {$E_6$};
\node[below] at (6,1) {$E_7$};
\node[below] at (7,1) {$L_x$};
\node[below] at (8,1) {$F_7$};
\node[below] at (9,1) {$F_6$};
\node[below] at (10,1) {$F_5$};
\node[below] at (11,1) {$F_4$};
\node[below] at (12,1) {$F_3$};
\node[below] at (13,1) {$F_2$};
\node[below] at (14,1) {$F_1$};

\node[left] at (1,2) {$L_y$};
\node[left] at (1,3) {$D_1$};
\node[above] at (7,3) {$D_2$};
\node[right] at (13,3) {$D_3$};
\node[right] at (13,2) {$L_z$};

\node[left] at (7,2) {$C$};

\filldraw[black]
(1,3) circle (2pt)
(1,2) circle (2pt)
(1,1) circle (2pt)
(0,1) circle (2pt)
(2,1) circle (2pt)
(3,1) circle (2pt)
(4,1) circle (2pt)
(5,1) circle (2pt)
(6,1) circle (2pt)
(7,1) circle (2pt)
(8,1) circle (2pt)
(9,1) circle (2pt)
(10,1) circle (2pt)
(11,1) circle (2pt)
(12,1) circle (2pt)
(13,1) circle (2pt)
(14,1) circle (2pt)
(13,2) circle (2pt)
(13,3) circle (2pt)
(7,2) circle (2pt)
(7,3) circle (2pt)
;

\end{tikzpicture}
\caption{Dual graph of the configuration of $(-2)$-curves lying in a general fibre.}
\label{fig:dualgraph}
\end{figure}
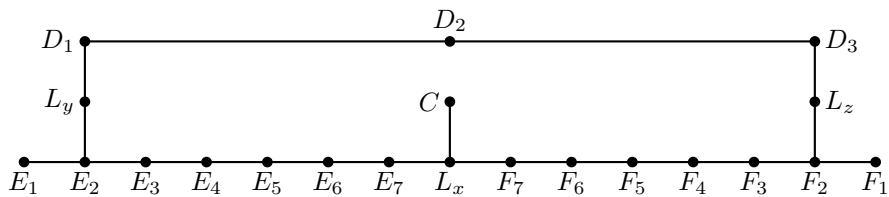

From this diagram, we can clearly see the four types of elliptic fibration on an $M_1$-polarized K3 surface (see \cite[Remark 7.11]{mslpk3s}). Note that there are three fibrations of each type (with the exception of type (4)), induced by the threefold rotational symmetry of Figure \ref{fig:dualgraph}.
\begin{enumerate}
\item $F := 2D_1 + 4L_y + 3E_1 + 6E_2 + 5E_3+4E_4+3E_5+2E_6+E_7$ and $F' := 2D_2 + 4L_z+ 3F_1 + 6F_2 + 5F_3+4F_4+3F_5+2F_6+F_7$ are a pair of fibres of type $\mathrm{II}^*$ and $C$ is one component of a fibre of type $\mathrm{I}_2$ (the other component is not part of this configuration). The line $L_x$ is a section and $D_2$ a bisection. This is the ``standard fibration'' of Clingher and Doran \cite[Proposition 3.10]{milpk3s}. Classes of sections and fibres in this fibration give rise to the $M_1 := H \oplus E_8 \oplus E_8 \oplus A_1$ polarization; from this we see that the $19$ classes $E_1,\ldots,E_6$, $F_1,\ldots,F_6$, $L_x$, $L_y$, $L_z$, $D_1$, $D_3$, $F$, and $F - C$ are a primitive set of generators for $M_1$. Coupled with the invariance under monodromy of the configuration from Figure \ref{fig:dualgraph}, this gives another proof of the fact (Proposition \ref{prop:X1}) that the family $\calX_1$ is an $M_1$-polarized family of K3 surfaces.
\item $L_y + E_1 + 2(E_2 + \cdots + E_7 + L_x + F_7 + \cdots + F_2) + F_1 + L_z$ is a fibre of type $\mathrm{I}_{12}^*$  and $D_2$ is one component of a fibre of type $\mathrm{I}_2$ (the other component is not part of this configuration). The curves $D_1$ and $D_3$ are sections and $C$ is a bisection. This is the ``alternate fibration'' of Clingher and Doran \cite[Proposition 3.10]{milpk3s}.
\item $D_2 + 2D_1 + 3L_y + 4E_2 + 3E_3 + 2 E_4 + E_5 + 2E_1$ is a fibre of type $\mathrm{III}^*$ and $L_z + F_1 + 2(F_2 + F_3+ F_4 + F_5 + F_6 + F_7 + L_x) + E_7 + C$ is a fibre of type $\mathrm{I}_6^*$. The curves $D_3$ and $E_6$ are sections.
\item $E_2 + \cdots + E_7 + L_x + F_7 + \ldots + F_2 + L_z + D_3 + D_2 + D_1 + L_y$ is a fibre of type $\mathrm{I}_{18}$, and $E_1$, $C$, $F_1$ are sections.
\end{enumerate}

We conclude this subsection with a brief discussion of involutions. There are two obvious involutions on the family \eqref{eq:2dimX1}, which both lift to involutions on the general fibres of a minimal resolution. The first is the involution $x \mapsto -x$. On a smooth $M_1$-polarized K3 surface, this is the antisymplectic involution that preserves the $M_1$-polarization and acts as $-\mathrm{Id}$ on the transcendental lattice (c.f. Remark \ref{rem:antisymplectic}).

More interestingly, there is also a second involution, exchanging $y$ and $z$. This involution lifts to a symplectic involution $\iota$ on a smooth $M_1$-polarized K3 surface $S$, which lies in the Mordell-Weil group of the ``alternate'' elliptic fibration (2). Work of Clingher and Doran  \cite[Proposition 3.13]{milpk3s} shows that $\iota$ defines a canonical Shioda-Inose structure (named for Shioda and Inose \cite{sk3s}, who first studied such structures) on $S$.

If we quotient $S$ by the involution $\iota$ and resolve the resulting singularities, \cite[Proposition 3.13]{milpk3s} and \cite[Section 3.2]{nfk3smmp} show that the resulting surface is the Kummer surface $\mathrm{Kum}(E \times E)$ associated to the product of an elliptic curve $E$ with itself. Moreover, the $j$-invariant of $E$ is given by $j(E) = -\frac{1}{\gamma}$, where $\gamma$ is the generalized functional invariant of $S$ (c.f.  Remark \ref{rem:jgamma}).

This process is reversible and canonically associates an elliptic curve $E$ to every $M_1$-polarized K3 surface. Furthermore, this construction can also be made to work for $M_1$-polarized families of K3 surfaces $\calX^U \to U$ (see \cite{cytfksapec}), although there are serious difficulties inherent in extending it further to singular fibres. The upshot is that any $M_1$-polarized family of K3 surfaces has a canonically associated elliptic surface; this explains the many parallels between the theory presented in this paper and the well-known theory of elliptic surfaces.

\subsection{An intrinsic definition of the Weierstrass form}\label{sec:canonical}

In this section we prove the following proposition, which provides an intrinsic derivation of the Weierstrass form \eqref{eq:2dimX1} that comes directly from the structure of the lattice $M_1$, rather than from abstract mirror considerations.

\begin{proposition} \label{prop:canonicalmodel} Let $S$ be an $M_1$-polarized K3 surface. By examining the elliptic fibrations on $S$ we may identify the $21$ curves from Figure \ref{fig:dualgraph}, up to the action of the symmetric group $S_3$. Such an identification induces a morphism $\phi\colon S \to \mathbb{P}^3[w,x,y,z]$, unique up to multiplication of $x$ by a nonzero scalar. If the generalized functional invariant of $S$ is not equal to $0$ or $\infty$, the image $\phi(S)$ is given by a quartic equation in the Weierstrass form \eqref{eq:2dimX1}.
\end{proposition}

\begin{remark} This result holds for individual $M_1$-polarized K3 surfaces $S$. Non-uniqueness of the variable $x$ means that there are significant technical obstructions to using this method to prove a result that works in families, akin to Theorem \ref{thm:weierstrassform}. This indeterminacy corresponds to the fact that we may freely rescale $x \mapsto tx$ and $\beta' \mapsto \frac{\beta'}{t^2}$ for $t \in \bC^*$ in Equation \eqref{eq:2dimX1} and obtain the same K3 surface.
\end{remark}

\begin{proof} Choose an identification of the $21$ curves from Figure \ref{fig:dualgraph} on $S$. Then consider the three divisors
\begin{align*}W &:= 2L_x + L_y + L_z + C + \sum_{i=1}^3D_i + E_1 + 2 \sum_{i=2}^7E_i + F_1 + 2 \sum_{i=2}^7 F_i, \\
Y &:= 4L_y + 3D_1 + 2D_2 + D_3 + 3E_1 + 6E_2 + 5E_3 + 4E_4 + 3E_5 + 2E_6 + E_7,\\
Z &:= 4L_z + D_1 + 2D_2 + 3D_3 + 3F_1 + 6F_2 + 5F_3 + 4F_4 + 3F_5 + 2F_6 + F_7.\\
\end{align*}
From the linear equivalence of the $\mathrm{III}^*$ and $\mathrm{I}_6^*$ fibres in elliptic fibration (3),  $W$ is linearly equivalent to $Y$, and from the linear equivalence of the two $\mathrm{II}^*$ fibres in elliptic fibration (1), $Y$ is linearly equivalent to $Z$. Thus $W$, $Y$, and $Z$ are cut out by linearly independent sections $w$, $y$, and $z$ of a common line bundle $\calL$ and these sections are unique up to scalar multiples.

A Riemann-Roch computation shows that $h^0(S,\calL) = 4$. We can extend $\langle w,y,z\rangle$ to a complete set of generators by adding an additional section. Let $X$ denote the divisor
\[X:= L_x + C + \sum_{i=1}^7E_i + \sum_{i=1}^7F_i + G,\]
where $G$ is an arbitrary smooth elliptic curve in the elliptic fibration (4). From the linear equivalence between $G$ and the $I_{18}$ fibre in fibration (4), we see that $X$ is linearly equivalent to $W$. The divisor $X$ is cut out by a section $x \in H^0(S,\calL)$ which is linearly independent of $\langle w,y,z\rangle$ and uniquely determined up to scalar multiples and transformations of the form $x + \mu w$, with $\mu \in \bC$ (the choice of $\mu$ is equivalent to choice of the curve $G$, which is not uniquely determined by the $M_1$-polarization).

The divisors $W$, $X$, $Y$, $Z$ contain no common points, so the line bundle $\calL$ is generated by its global sections. By Riemann-Roch we see that $\calL$ induces a map $\phi\colon S \to \bP^3$ whose image is a quartic hypersurface, and we may identify $w$, $x$, $y$, $z$ with homogeneous coordinates. The map $\phi$ satisfies:
\begin{enumerate}[(i)]
\item the divisors $C$, $D_i$, $E_j$, and $F_k$ are contracted for all $i,j,k$;
\item the image of $W$ is the hyperplane $\{w = 0\}$, and the intersection of $\phi(S)$ with this hyperplane is $\{x^2yz = 0\}$;
\item the images of $Y$ and $Z$ are the hyperplanes $\{y=0\}$ and $\{z=0\}$ respectively, and the restriction of $\phi(S)$ to either hyperplane is $\{w^4 = 0\}$;
\item the image of $X$ is the hyperplane $\{x = 0\}$ and the restriction of $\phi(S)$ to this hyperplane is the union of the line $\{w = 0\}$ and a smooth cubic curve.
\end{enumerate}
Using these properties, we find that $\phi(S)$ is a hypersurface with the general form
\[a_1w^4 + a_2w^2yz + a_3wxyz + a_4wy^2z + a_5wyz^2 + a_6x^2yz = 0,\]
where $a_i \in \bC$. Moreover, (ii) above gives $a_6 \neq 0$, (iii) gives $a_1 \neq 0$, and (iv) gives $a_4,a_5 \neq 0$. Using this, the coordinate change $x \mapsto x - \frac{a_3}{2a_6}w$ (which uniquely specifies the elliptic curve $G$) gives
\begin{equation} \label{eq:phiimage} a_1w^4 + \big(a_2 - \tfrac{a_3^2}{4a_6}\big)w^2yz + a_4wy^2z + a_5wyz^2 + a_6x^2yz = 0.\end{equation}

If we assume that $a_2 - \tfrac{a_3^2}{4a_6} \neq 0$, then the change of coordinates  $(w,y,z) \mapsto (\frac{a_6w}{a_2a_6-\frac{1}{4}a_3^2},\frac{y}{a_4}, \frac{z}{a_5})$ and rescaling by $\frac{a_4a_6}{a_6}(a_2a_6-\frac{1}{4}a_3^2)$ gives
\[ \frac{a_1a_4a_5a_6^3}{(a_2a_6 - \tfrac{1}{4}a_3^2)^3} w^4 + w^2yz + wy^2z + wyz^2 + (a_2a_6 - \tfrac{1}{4}a_3^2)x^2yz = 0\]
This is clearly in the form of Equation \eqref{eq:2dimX1}, with $\beta'$ and $\gamma$ given by
\[\beta' = (a_2a_6 - \tfrac{1}{4}a_3^2) \notin \{0,\infty\}, \quad\quad \gamma =  \frac{a_1a_4a_5a_6^3}{(a_2a_6 - \tfrac{1}{4}a_3^2)^3} \notin \{0,\infty\}.\]
After this coordinate change the sections $w$, $y$, $z$ are defined uniquely, so the map $\phi$ is unique up to multiplication of $x$ by a nonzero scalar. 

It just remains to deal with the case $a_2 - \tfrac{a_3^2}{4a_6} = 0$, which is treated by the following lemma; compare also Proposition \ref{prop:infinityfibres2}, which describes the same surfaces in the context of degenerations.\end{proof}

\begin{lemma} \label{lem:gammainfinity} If $a_2 - \tfrac{a_3^2}{4a_6} = 0$, then $S$ has generalized functional invariant $\infty$ and $\phi(S)$ is isomorphic to the surface
\[\{\tfrac{1}{27}w^4 + wy^2z+wyz^2+x^2yz = 0\} \subset \bP^3[w,x,y,z].\]
\end{lemma}
\begin{proof} The expression for $\gamma$ above shows that the generalized functional invariant of $S$ is $\infty$ in this case. To show that $\phi(S)$ has the required form, set $a_2 - \tfrac{a_3^2}{4a_6} = 0$ in Equation \eqref{eq:phiimage}, and rescale the resulting equation so that $a_1a_4a_5a_6^3 = \frac{1}{27}$. Changing coordinates $(w,y,z) \mapsto (a_6w,\tfrac{y}{a_4}, \tfrac{z}{a_5})$ and rescaling by $\frac{a_4a_5}{a_6}$ proves the lemma.
\end{proof}

\section{Singular fibres}\label{sec:singularfibres}

By Theorem \ref{thm:weierstrassform}, any threefold fibred non-isotrivially by $M_1$-polarized K3 surfaces has a Weierstrass model and, by the results of Section \ref{sec:genfibres} we know how to resolve the singularities in the generic fibres of such models. It remains to study the local behaviour around the singular fibres. In this section we will classify these singular fibres up to birational equivalence and exhibit special local models for each.

Let $\beta'(\delta)$ and $\gamma(\delta)$ be meromorphic functions on the open complex unit disc $\Delta$ such that $\beta'(\delta) \notin \{0,\infty\}$ and $\gamma(\delta) \notin \{0,-1,\infty\}$ for all $\delta$ in the punctured disc $\Delta^* = \Delta \setminus \{0\}$. Let $\overline{\pi}\colon \overline{\calX} \to \Delta$ be a threefold defined by Equation \eqref{eq:2dimX1} with $\beta' = \beta'(\delta)$ and $\gamma = \gamma(\delta)$ and let $\pi \colon \calX \to \Delta$ denote a minimal resolution of $\overline{\calX}$. Note that the fibres of $\calX$ over $\Delta^*$ are smooth $M_1$-polarized K3 surfaces.

Suppose that $\beta'$ ramifies to order $b$ and $\gamma$ ramifies to order $d$ at $\delta = 0$. It then follows from Theorem \ref{thm:Lclassification} and Proposition \ref{prop:plusminus} that, up to shrinking $\Delta$ and a birational map affecting only the fibre over $\delta = 0$, the threefold $\calX$ is determined by the following three pieces of data:
\begin{enumerate}[(i)]
\item The point $\gamma(0)$;
\item The order of ramification $d$ of $\gamma$ at $0$;
\item The sign of the determinant $\det(\rho(\sigma))$ of the generalized homological invariant $\rho$ of $\calX$, where $\sigma \in  \pi_1(\Delta^*)$ is a simple closed loop around $\delta = 0$; this may be computed explicitly from $b$ and $d$ by Proposition \ref{prop:generalmonodromy}.
\end{enumerate}

The results of this section are summarized by the following theorem. We prove it via a case-by-case analysis in subsections \ref{subsec:gammageneral}--\ref{subsec:gammainfinity}. This analysis also contains a detailed description of the various fibre types, along with figures illustrating some of them.

\begin{theorem} \label{thm:singularfibres} With notation and assumptions as above, we may perform a birational modification $\calX \dashrightarrow \calX'$ that affects only the fibre over $0 \in \Delta$, so that the singularities of $\calX'$ are at worst terminal and the canonical sheaf $\omega_{\calX'}$ is trivial. The central fibre of $\calX'$ is one of the possibilities listed in Table \ref{tab:fibres}.
\end{theorem}

Table \ref{tab:fibres} classifies the singular fibres that appear in our models $\pi'\colon \calX' \to \Delta$ and gives some of their basic properties. In this table: ``Type'' gives a name for the singular fibre, in analogy with the singular fibres appearing in elliptic surfaces; ``$\gamma$'' gives the value of $\gamma(0)$; ``$d$'' is the order of ramification of $\gamma$ at $0$; ``Components'' is the number of components of the fibre; ``$\rho$'' is the matrix of the generalized homological invariant $\rho(\sigma)$ (up to conjugation); ``det'' is its determinant; $\omega$ is a primitive sixth root of unity; and fibres marked ``(singular)'' have some additional isolated singularity behaviour, either in the fibre or in the threefold $\calX'$. Finally, $R$ and $S$ are two invariants associated to the singular fibre that will be explained further in Section \ref{sec:invariants}. 

\begin{table}
\begin{center}
\begin{tabular}{|c|c|c|c|c|c|c|c|}
\hline
Type & $\gamma$ & $d$ & Components & $\rho$ & det & $R$ & $S$ \\
\hline
  & $\neq 0,-1,\infty$ &  $d \geq 1$ & $1$ & & & &  \\ 
\cline{2-3}
$\mathrm{I}_0$& $\infty$ & $0 \pmod 3$  &   & $\mathrm{Id}$ & $1$ &$0$ & $0$\\
\cline{2-4}
& $-1$ & $0 \pmod 2$ & $1$ (singular) & & & & \\
\hline
& $\neq 0,-1,\infty$ & $d \geq 1$   & $11$ && & & \\ 
\cline{2-3}
$\mathrm{I}_0^*$ & $\infty$ & $0 \pmod 3$ &   & $-\mathrm{Id}$ & $-1$ & $3$ & $\frac{1}{2}$  \\
\cline{2-4}
& $-1$ & $0 \pmod 2$ & $11$ (singular) & & & & \\
\hline
$\mathrm{I}_d$ & $0$ &  $d \geq 1$ & $d^2 + 2$ & $\begin{psmallmatrix}1 & d & \frac{d(d-1)}{2} \\ 0 & 1 & d \\ 0 & 0 & 1 \end{psmallmatrix}$ & $1$ & $2$ & $0$ \\
\hline
$\mathrm{I}_d^*$ & $0$ & $d \geq 1$ & $2d^2 + 9d + 10$ & $\begin{psmallmatrix}-1 & -d & -\frac{d(d-1)}{2} \\ 0 & -1 & -d \\ 0 & 0 & -1 \end{psmallmatrix}$ & $-1$ & $3$ & $\frac{1}{2}$ \\
\hline
$\mathrm{III}$ & $-1$ & $1 \pmod 2$ & $1$ (singular) & $\begin{psmallmatrix}-1 & 0 & 0 \\ 0 & 1 & 0 \\ 0 & 0 & 1 \end{psmallmatrix}$ & $-1$ & $1$ & $0$ \\
\hline 
$\mathrm{III}^*$ & $-1$ & $1 \pmod 2$ & $11$ (singular) & $\begin{psmallmatrix}1 & 0 & 0 \\ 0 & -1 & 0 \\ 0 & 0 & -1 \end{psmallmatrix}$ & $1$ & $2$ & $\frac{1}{2}$ \\
\hline
$\mathrm{II}^*$ & $\infty$ & $1 \pmod 3$ & $53$ & $\begin{psmallmatrix}-1 & 0 & 0 \\ 0 & \omega & 0 \\ 0 & 0 & \omega^5 \end{psmallmatrix}$ & $-1$ & $3$ & $\frac{5}{6}$ \\
\hline
$\mathrm{IV}^*$ & $\infty$ & $2 \pmod 3$ & $22$ & $\begin{psmallmatrix}1 & 0 & 0 \\ 0 & \omega^2 & 0 \\ 0 & 0 & \omega^4 \end{psmallmatrix}$ & $1$ & $2$ & $\frac{2}{3}$ \\
\hline
$\mathrm{IV}$ & $\infty$ & $1 \pmod 3$ & $6$ & $\begin{psmallmatrix}1 & 0 & 0 \\ 0 & \omega^4 & 0 \\ 0 & 0 & \omega^2 \end{psmallmatrix}$ & $1$ & $2$ &$\frac{1}{3}$ \\
\hline
$\mathrm{II}$ & $\infty$ & $2 \pmod 3$ & $3$ & $\begin{psmallmatrix}-1 & 0 & 0 \\ 0 & \omega^5 & 0 \\ 0 & 0 & \omega \end{psmallmatrix}$ & $-1$ & $3$ & $\frac{1}{6}$ \\
\hline
\end{tabular}
\end{center}
\caption{Summary of singular fibres}
\label{tab:fibres}
\end{table}

\begin{remark} Any two such models $\pi' \colon \calX' \to \Delta$ are related by a birational map affecting only the fibre over $\delta = 0$, which must be an isomorphism in codimension $1$ by \cite[Proposition 3.54]{bgav}. The classification given by Theorem \ref{thm:singularfibres} only holds up to such birational maps. These maps do not affect any of the properties listed in Table \ref{tab:fibres}, but may alter the configurations of components shown in the various illustrative figures in this section.
\end{remark}

To prove Theorem \ref{thm:singularfibres}, we construct the birational models $\pi'\colon \calX' \to \Delta$ explicitly. To do this, we first build singular families $\overline{\pi} \colon \overline{\calX} \to \Delta$ realizing all possibilities for the data (i)-(iii) above by choosing appropriate meromorphic functions $\beta'(\delta)$ and $\gamma(\delta)$ in Equation \eqref{eq:2dimX1}. When we can take $\beta'(0) \notin \{0,\infty\}$, we may simplify matters by using Equation \eqref{eq:X1} instead of Equation \eqref{eq:2dimX1}. 

The cases below come in pairs for each value of $\gamma(0)$, corresponding to the two possible values of $\det(\rho(\sigma))$. In each case we first use Equation \eqref{eq:X1} to treat the case where $\beta'(0) \notin \{0,\infty\}$; this gives us one of the possibilities for $\det(\rho(\sigma))$. We then proceed to Equation \eqref{eq:2dimX1} and allow $\beta'$ to have a zero or pole, which flips the sign of $\det(\rho(\sigma))$ and gives the other possibility. In each case we perform explicit birational modifications to obtain a model $\calX'$ with the required properties. 

By Section \ref{sec:genfibres}, away from $\delta = 0$ the threefolds $\overline{\calX}$ will have precisely four curves of singularities that form sections of the fibration: two curves of $A_7$ singularities, denoted $C_1^{A_7} := \{w= x = y = 0\}$ and $C_2^{A_7} := \{w= x = z = 0\}$; a curve of $A_3$'s,  denoted $C^{A_3} := \{w=y = z = 0\}$; and a curve of $A_1$'s, denoted $C^{A_1} := \{w=x = y+z = 0\}$. Away from $\delta = 0$ these may be crepantly resolved by blowing up in the usual way; their behaviour at $\delta = 0$ will be dealt with case-by-case. 

\subsection{A legend for interpreting figures} \label{sec:legend} 

In the remainder of this section, we will display partial or full resolutions of many of the singular fibre types in Figures \ref{fig:mn}--\ref{fig:II}. Here is a legend for interpreting these figures.
\begin{itemize}
\item Solid lines denote double curves and the boundaries of components; components should be assumed to meet normally along double curves and at triple points unless otherwise stated.
\item In some cases the threefold is singular along a double curve or a component boundary; when this occurs the line is drawn bold and labelled with the corresponding singularity.
\item In some cases the threefold contains a curve of $A_1$ singularities that lies wholly within a component; such curves are drawn dashed.
\item Dotted lines mark the edges of local diagrams.
\item Bold points give the intersection of a fibre with the four curves of singularities that form sections of the fibration: $C_1^{A_7}$ and $C_2^{A_7}$ are always drawn in the bottom corners, $C^{A_3}$ is always at the top, and $C^{A_1}$ is always the bottom middle.
\item Points labelled by a letter correspond to more complicated threefold singularities, which will be described individually as they occur.
\item Faint gray lines radiating from the edges of a diagram denote exceptional divisors arising from the curves  $C_1^{A_7}$, $C_2^{A_7}$, $C^{A_3}$, and $C^{A_1}$; as each of these curves forms a section of the fibration, these exceptional divisors are horizontal and do not form part of the singular fibre under examination.
\item Unshaded components should be visualized as lying flat on the page, with gray shaded components protruding out of the page towards the reader.
\item Integers labelling components denote multiplicities.
\end{itemize}

For example, Figure \ref{fig:I0*unresolved} shows an unresolved singular fibre containing three irreducible components. It contains two curves of $A_1$ singularities: one along the bottom edge and one given by the dashed curve in the component $\{x = 0\}$. The curves $C_1^{A_7}$, $C_2^{A_7}$, $C^{A_3}$, and $C^{A_1}$, meet this fibre in the four bold points, and the points labelled $P$ and $Q$ are more complicated threefold singularities. Figure \ref{fig:I0*} shows a resolution of this same fibre. It contains eleven components, ten of which lie in the plane of the page and one of which protrudes outwards. There are ten components of multiplicity $1$ and one of multiplicity $2$.

\subsection{Singularities of type \texorpdfstring{$(m,n)$}{(m,n)}}

A particular type of singularity will arise repeatedly in our analysis of the singular fibres. It will be very useful to have an understanding of this type of singularity before we embark on this analysis.

\begin{definition} Let $m \geq n \geq 0$ be integers with $m \geq 2$. A point $P$ on a threefold $\calX$ is called a \emph{singularity of type $(m,n)$} if $P \in \calX$ is locally analytically isomorphic to 
\[0 \in \{s^mt^n + stu + uv^2 = 0\} \subset \bA^4[s,t,u,v].\]
\end{definition}

If $n > 0$, a singularity of type $(m,n)$ lies at the confluence of three curves of singularities of types $A_{2m-1}$, $A_{2n-1}$, and $A_1$. If $n = 0$, instead there are two curves of singularities, of types $A_{2m-1}$ and $A_1$.

\begin{lemma} \label{lem:singularity} Consider a singularity of type $(m,n)$, given locally by the point $(0,0,0,0)$ in the hypersurface
\[\{s^mt^n + stu + uv^2 = 0\} \subset \bA^4[s,t,u,v].\]
After crepantly resolving the singularities of this hypersurface, the point $(0,0,0,0)$ gives rise to $(m+n-2)$ exceptional components in addition to the exceptional components arising from the curves of singularities. Moreover, after the resolution,
\begin{itemize}
\item if $n > 0$, all exceptional components and the strict transforms of the four hyperplanes $\{s = u = 0\}$, $\{s = v = 0\}$, $\{t = u = 0\}$, and $\{t = v = 0\}$ meet normally, and
\item if $n = 0$, all exceptional components and the strict transforms of the two hyperplanes $\{s = u = 0\}$ and $\{s = v = 0\}$ meet normally.
\end{itemize}
\end{lemma}
\begin{proof} This is a lengthy explicit computation. If $(m,n) \neq (2,0)$, we begin by blowing-up the curve of $A_{2m-1}$ singularities, followed by the curve of $A_1$ singularities. This gives rise to four exceptional components: two from the curve of $A_{2m-1}$'s, one from the curve of $A_1$'s, and one from the point $(0,0,0,0)$. If $(m,n) = (2,0)$ we perform the same blow-ups in the opposite order, obtaining three exceptional components. The resulting threefolds are depicted in Figure \ref{fig:mn}. In this figure, the strict transforms of hyperplanes in the original threefold are labelled by the pair of variables that vanish on that component (so the strict transform of the component ${s=v = 0}$ is labelled $(s,v)$, etc.), and the other components are exceptional. 

\begin{figure}
  \begin{center}
  \begin{subfigure}{0.45\textwidth}
  \begin{center}
\begin{tikzpicture}[scale=0.7]
\filldraw[lightgray](5,2)--(3.5,4.6)--(3.5,2.6)--(5,0)--(5,2);
\draw (0,1.5)--(2,1.5)--(2,0);
\draw (2,1.5)--(2.5,2)--(2.25,2.43)--(0,2.43);
\draw(2.5,2)--(3.85,2);
\draw[dotted](3.5,2)--(5,2);
\draw(5,2)--(7,2);
\draw(2.25,2.43)--(3.5,4.6)--(3.5,2.6);
\draw[dotted] (3.5,2.6)--(5,0);
\draw[very thick] (3.5,6)--(3.5,4.6)--(5,2)--(5,0);
\draw [dotted] (0,0)--(0,6)--(7,6)--(7,0)--(0,0);

\node[above right] at (4.95,-0.1) {$A_{2m-3}$};
\node[below right] at (3.45,6) {$A_{2n-1}$};
\node[above right] at (4.25,3) {$A_{m+n-3}$};
\node at (1,4.7) {$(t,v)$};
\node at (6,4.7) {$(t,u)$};
\node at (1,0.7) {$(s,v)$};
\node at (6,1) {$(s,u)$};
\end{tikzpicture}
\caption{$n \geq 2$.}
\label{fig:n>1}
\end{center}
\end{subfigure}
\begin{subfigure}{0.45\textwidth}
  \begin{center}
\begin{tikzpicture}[scale=0.7]
\filldraw[lightgray](5,2)--(3.5,4.6)--(3.5,2.6)--(5,0)--(5,2);
\draw (0,1.5)--(2,1.5)--(2,0);
\draw (2,1.5)--(2.5,2)--(2.25,2.43)--(0,2.43);
\draw(2.5,2)--(3.85,2);
\draw(2.5,2)--(3.85,2);
\draw[dotted](3.5,2)--(5,2);
\draw(5,2)--(7,2);
\draw(2.25,2.43)--(3.5,4.6)--(3.5,2.6);
\draw[dotted] (3.5,2.6)--(5,0);
\draw(3.5,6)--(3.5,4.6);
\draw[very thick] (3.5,4.6)--(5,2)--(5,0);

\draw[dashed,domain=0:1.4] plot ({3.5-\x*\x*\x},{\x+4.6});

\draw [dotted] (0,0)--(0,6)--(7,6)--(7,0)--(0,0);

\node[above right] at (4.95,-0.1) {$A_{2m-3}$};
\node[above right] at (4.25,3) {$A_{m-2}$};
\node at (1,4.7) {$(t,v)$};
\node at (6,4.7) {$(t,u)$};
\node at (1,0.7) {$(s,v)$};
\node at (6,1) {$(s,u)$};
\end{tikzpicture}
\caption{$m \geq 3$, $n = 1$.}
\label{fig:n=1m>2}
\end{center}
\end{subfigure}
\bigskip 

\begin{subfigure}{0.45\textwidth}
  \begin{center}
\begin{tikzpicture}[scale=0.7]
\filldraw[lightgray](5,2)--(2.75,3.3)--(2.75,1.3)--(5,0)--(5,2);
\draw (0,1.5)--(2,1.5)--(2,0);
\draw (2,1.5)--(2.5,2)--(2.25,2.43)--(0,2.43);
\draw(2.5,2)--(2.75,2);
\draw[dotted](2.75,2)--(5,2);
\draw(5,2)--(7,2);
\draw(2.25,2.43)--(3.5,4.6);
\draw(3.5,6)--(3.5,4.6)--(5,2);
\draw [very thick] (5,2)--(5,0);
\draw (5,2)--(2.75,3.3)--(2.75,1.3);
\draw [dotted] (2.75,1.3)--(5,0);
\draw [dashed] (2.75,3.3)--(0.05,6);

\draw [dotted] (0,0)--(0,6)--(7,6)--(7,0)--(0,0);

\node[above right] at (4.95,-0.1) {$A_1$};
\node at (1,4) {$(t,v)$};
\node at (6,4.7) {$(t,u)$};
\node at (1,0.7) {$(s,v)$};
\node at (6,1) {$(s,u)$};
\end{tikzpicture}
\caption{$m =2$, $n = 1$.}
\label{fig:n=1m=2}
\end{center}
\end{subfigure}
\begin{subfigure}{0.45\textwidth}
  \begin{center}
\begin{tikzpicture}[scale=0.7]
\filldraw[lightgray](5,2)--(3.5,4.6)--(3.5,2.6)--(5,0)--(5,2);
\draw (0,1.5)--(2,1.5)--(2.5,2)--(2.5,4.6);
\draw (0,4.6)--(3.5,4.6);
\draw (2,1.5)--(2,0);
\draw(2.5,2)--(3.85,2);
\draw[dotted](3.5,2)--(5,2);
\draw (3.5,4.6)--(3.5,2.6);
\draw[dotted] (3.5,2.6)--(5,0);
\draw[very thick] (5,2)--(5,0);
\draw [very thick] (3.5,4.6)--(5,2);
\draw (5,2)--(7,2);
\draw[dotted] (0,4.6)--(0,0)--(7,0)--(7,2);

\node[above right] at (4.95,-0.1) {$A_{2m-3}$};
\node[above right] at (4.25,3) {$A_{m-3}$};
\node at (1,0.7) {$(s,v)$};
\node at (6,1) {$(s,u)$};
\end{tikzpicture}
\caption{$m \geq 3$, $n =0$.}
\label{fig:n=0m>2}
\end{center}
\end{subfigure}
\bigskip

\begin{subfigure}{0.45\textwidth}
 \begin{center}
\begin{tikzpicture}[scale=0.7]
\filldraw[lightgray](5,2)--(3.5,4.6)--(3.5,2.6)--(5,0)--(5,2);
\draw (0,4.6)--(7,4.6);
\draw (2,2)--(2,0);
\draw(0,2)--(3.85,2);
\draw[dotted](3.5,2)--(5,2);
\draw (3.5,4.6)--(3.5,2.6);
\draw[dotted] (3.5,2.6)--(5,0);
\draw[very thick] (5,2)--(5,0);
\draw (3.5,4.6)--(5,2);
\draw (5,2)--(7,4.6);
\draw[dotted] (0,4.6)--(0,0)--(7,0)--(7,4.6);

\node[above right] at (4.95,-0.1) {$A_{1}$};
\node at (1,1) {$(s,v)$};
\node at (6,1) {$(s,u)$};
\end{tikzpicture}
\caption{$m =2$, $n =0$.}
\label{fig:n=0m=2}
\end{center}
\end{subfigure}
\caption{Partial resolution of a singularity of type $(m,n)$.}
\label{fig:mn}
\end{center}
\end{figure}
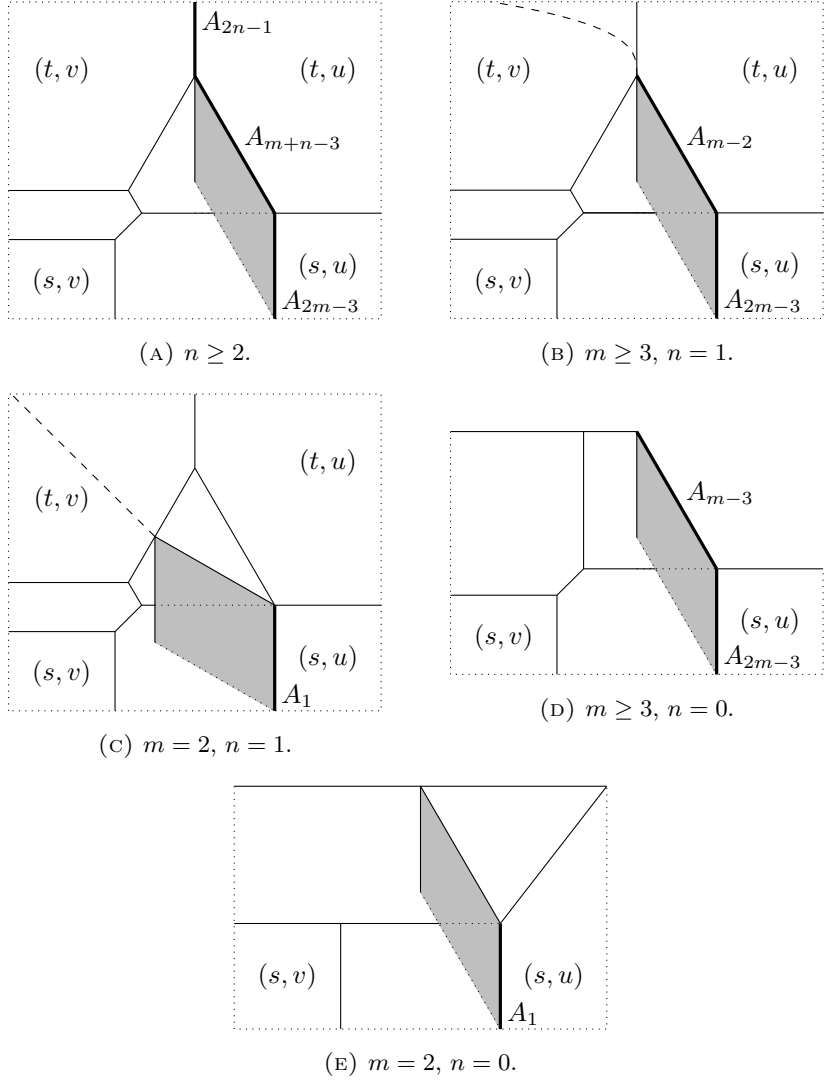

From this point, each remaining curve of singularities may be blown up individually to obtain the resolution; there is no unusual behaviour at the points where the curves of singularities meet.
\end{proof}

\subsection{Fibres with \texorpdfstring{$\gamma \notin \{0, -1, \infty\}$}{gamma not equal to 0, -1, infinity}}\label{subsec:gammageneral} We now begin our case-by-case analysis of the singular fibres, beginning with those fibres with $\gamma(0) \notin \{0, -1, \infty\}$.

 \subsubsection{Type \texorpdfstring{$\mathrm{I}_0$}{I0}} Data: $\gamma \notin \{0,-1,\infty\}$, $d \geq 1$, $\det(\rho) = 1$.  
 
These data are realized by the family $\overline{\calX} \to \Delta$ given by restricting Equation \eqref{eq:X1} to an appropriate disc $\Delta$ that does not contain $\{0,-1,\infty\}$. Section \ref{sec:genfibres} shows that the threefold $\overline{\calX}$ will be singular only along the four non-intersecting curves $C_1^{A_7}$, $C_2^{A_7}$, $C^{A_3}$, and $C^{A_1}$. They may be resolved in the usual way to obtain the smooth model $\calX'$. The generalized homological invariant evaluated on any loop is the identity. A cover totally ramified over the central fibre retains the same behaviour and the same type.

\subsubsection{Type \texorpdfstring{$\mathrm{I}_0^*$}{I0*}}\label{sec:I0*} Data: $\gamma \notin \{0,-1,\infty\}$, $d \geq 1$, $\det(\rho) = -1$. 

Let $c \notin \{0,-1,\infty\}$ be an arbitrary constant and $\delta \in \Delta$. These data are realized by Equation \eqref{eq:2dimX1} when $\beta'(\delta) = \frac{1}{\delta}$ and $\gamma =  (\delta^d + c)$. Clearing denominators gives
\begin{equation} \label{eq:I0*}\frac{\delta(\delta^d + c)}{27} w^4+ \delta w^2yz  + \delta wy^2z + \delta wyz^2  + x^2yz = 0,\end{equation}
which produces a family $\overline{\calX} \to \Delta$ with fibre over $\delta = 0$ as displayed in Figure \ref{fig:I0*unresolved}.  Note that the dotted line of $A_1$ singularities in this figure is a smooth elliptic curve in the component $\{x = 0\}$, defined by the equation $\frac{c}{27}w^3 + wyz + y^2z + yz^2 = 0$.

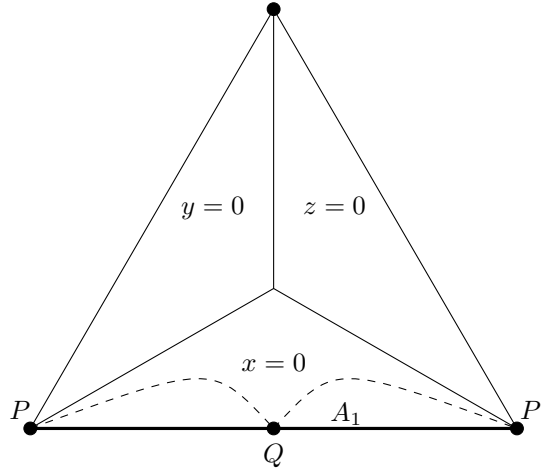
\begin{figure}
\begin{center}
\begin{tikzpicture}[scale=0.8]
\draw[very thick] (0,0)--(8,0);
\draw (0,0)--(4,6.95)--(8,0);
\draw (0,0)--(4,2.32)--(8,0);
\draw (4,2.32)--(4,6.95);
\draw [dashed] (4,0) .. controls (5,1.1) .. (8,0);
\draw [dashed] (0,0) .. controls (3,1.1) .. (4,0);

\filldraw[black]
(0,0) circle (3pt)
(4,0) circle (3pt)
(8,0) circle (3pt)
(4,6.95) circle (3pt)
;

\node[above] at (5.2,-0.1) {$A_1$};
\node[right] at (7.9,0.3) {$P$};
\node[right] at (-0.5,0.3) {$P$};
\node[below] at (3,4) {$y=0$};
\node[below] at (5,4) {$z=0$};
\node at (4,1.1) {$x = 0$};
\node[below] at (4,-0.1) {$Q$};
\end{tikzpicture}
\caption{Fibre over $\delta = 0$ in Equation \eqref{eq:I0*}.}
\label{fig:I0*unresolved}
\end{center}
\end{figure}

\begin{proposition}\label{prop:I0*} The threefold defined by Equation \eqref{eq:I0*} admits a crepant resolution to a smooth threefold. The central fibre of this threefold contains $11$ components which meet normally. The generalized homological invariant evaluated on a simple closed loop around $\delta = 0$ is conjugate to $-\mathrm{Id}$. This fibre type is called $\mathrm{I}_0^*$.
\end{proposition}

\begin{proof}Away from the points marked $P$ and $Q$ in Figure \ref{fig:I0*unresolved}, we have smooth curves of $A_i$ singularities which may be crepantly resolved by blowups in the usual way. 

The point marked $Q$ is the transverse intersection of three curves of $A_1$ singularities (two in the fibre and the curve $C^{A_1}$). This configuration may be simultaneously resolved by the toric method of \cite[Section 1]{bgd7} as follows. Local analytically, $Q \in \overline{\calX}$ has the form $0 \in \{stu = v^2\} \subset \bA^4[s,t,u,v]$. This corresponds to an affine toric variety with cone generated by the rays $(1, 0, 0)$, $(0, 1, 0)$, $(-1, -1, 2)$ and dual cone generated by the rays $(0, 0, 1)$, $(2, 0, 1)$, $(0, 2, 1)$. The singularity can be resolved by refining this dual cone until the minimal generators of each subcone form a $\bZ$-basis for $\bZ^3$, then taking the associated toric variety. Such a refinement can be achieved by adding three additional rays generated by $(1,0,1)$, $(0,1,1)$, and $(1,1,1)$, along with three faces spanned by pairs of these. As all rays added are generated by interior lattice points on a height $1$ slice, the resulting resolution is crepant.

The points marked $P$ are singularities of type $(4,1)$ which we may resolve by Lemma \ref{lem:singularity}. The components that pass through them are the hyperplanes from this lemma, so the resulting configuration of divisors meets normally.

The resulting threefold is smooth with trivial canonical sheaf and its central fibre contains $11$ components which meet normally; this fibre is shown in Figure \ref{fig:I0*}. We note for future reference that the gray shaded component is elliptic ruled, but the other components are rational. By Proposition \ref{prop:generalmonodromy}, the generalized homological invariant evaluated on a simple closed loop around $\delta = 0$ must be $-\mathrm{Id}$. \end{proof}

\begin{figure}
  \begin{center}
\begin{tikzpicture}[scale=0.33]
%Central "triangles"
\draw (27.5,3.15)--(18.6,18.5)--(17.3,19.25)--(16,18.5)--(14.7,19.25)--(13.4,18.5)--(4.5,3.15);

%Vertical lines at the top
\draw (16,9.2)--(16,18.5);
\draw [lightgray] (18.6,18.5)--(18.6,20.5);
\draw [lightgray] (17.3,19.25)--(17.3,20.5);
\draw [lightgray] (14.7,19.25)--(14.7,20.5);
\draw [lightgray] (13.4,18.5)--(13.4,20.5);

%left corner
\draw (3.5,2)--(4.5,3.15)--(5.5,3.15)--(16,9.2);
\draw (5.5,3.15)--(6.5,1.4);
\draw (3.5,2)--(4.5,0.25);

\draw (4.5,0.25)--(5.85,0.25)--(6.5,1.4)--(8.5,0.25);
\draw (5.85,0.25)--(6.5,-0.9);

\draw (6.5,-0.9)--(7.85,-0.9)--(8.5,0.25)--(10.5,-0.9);
\draw (7.85,-0.9)--(8.5,-2.05);

\draw (8.5,-2.05)--(9.85,-2.05)--(10.5,-0.9);
\draw (9.85,-2.05)--(10.5,-3.2);

\draw (10.5,-3.2)--(14.5,-5.5);

\draw [lightgray] (3.5,2)--(1.5,0.85);
\draw [lightgray] (4.5,0.25)--(2.5,-0.9);
\draw [lightgray] (6.5,-0.9)--(4.5,-2.05);
\draw [lightgray] (8.5,-2.05)--(6.5,-3.2);
\draw [lightgray] (10.5,-3.2)--(8.5,-4.35);
\draw [lightgray] (12.5,-4.35)--(10.5,-5.5);
\draw [lightgray] (14.5,-5.5)--(12.5,-6.65);

%right corner
\draw (28.5,2)--(27.5,3.15)--(26.5,3.15)--(16,9.2);
\draw (26.5,3.15)--(25.5,1.4);
\draw (28.5,2)--(27.5,0.25);

\draw (27.5,0.25)--(26.15,0.25)--(25.5,1.4)--(23.5,0.25);
\draw (26.15,0.25)--(25.5,-0.9);

\draw (25.5,-0.9)--(24.15,-0.9)--(23.5,0.25)--(21.5,-0.9);
\draw (24.15,-0.9)--(23.5,-2.05);

\draw (23.5,-2.05)--(22.15,-2.05)--(21.5,-0.9);
\draw (22.15,-2.05)--(21.5,-3.2);

\draw (21.5,-3.2)--(17.5,-5.5);

\draw [lightgray] (28.5,2)--(30.5,0.85);
\draw [lightgray] (27.5,0.25)--(29.5,-0.9);
\draw [lightgray] (25.5,-0.9)--(27.5,-2.05);
\draw [lightgray] (23.5,-2.05)--(25.5,-3.2);
\draw [lightgray] (21.5,-3.2)--(23.5,-4.35);
\draw [lightgray] (19.5,-4.35)--(21.5,-5.5);
\draw [lightgray] (17.5,-5.5)--(19.5,-6.65);

%Bottom components
\draw (10.5,-0.9)--(21.5,-0.9);
\draw (14.5,-5.5)--(17.5,-5.5);

%Raised components
\filldraw[white] (16.5,-7)--(16.5,-2.7)--(16,-3.2)--(15.5,-5.5)--(15.5,-7)--(16.5,-7);
\filldraw[lightgray] (16,-3.2)--(15.5,-0.9)--(16.5,0.1)--(16.5,-2.7)--(16,-3.2);
\draw [lightgray] (15.5,-7)--(15.5,-5.5);
\draw [lightgray] (16.5,-7)--(16.5,-2.7);
\draw (15.5,-5.5)--(16,-3.2)--(15.5,-0.9);
\draw (16.5,-2.7)--(16.5,0.1);
\draw (16,-3.2)--(16.5,-2.7);

\filldraw[white] (5.25,2.4)--(1.75,0.4)--(2.25,-0.45)--(4.25,0.7)--(5.25,2.4);
\filldraw[lightgray] (5.25,1.85)--(6.25,1.85)--(6.25,2.975)--(5.25,2.4)--(5.25,1.85);
\draw [lightgray] (5.25,2.4)--(1.75,0.4);
\draw [lightgray] (4.25,0.7)--(2.25,-0.45);
\draw (4.25,0.7)--(5.25,1.85)--(6.25,1.85);
\draw (5.25,2.4)--(6.25,2.975);
\draw (5.25,2.4)--(5.25,1.85);

\filldraw[white] (26.75,2.4)--(30.25,0.4)--(29.75,-0.45)--(27.75,0.7)--(26.75,2.4);
\filldraw[lightgray] (26.75,1.85)--(25.75,1.85)--(25.75,2.975)--(26.75,2.4)--(26.75,1.85);
\draw [lightgray] (26.75,2.4)--(30.25,0.4);
\draw [lightgray] (27.75,0.7)--(29.75,-0.45);
\draw (27.75,0.7)--(26.75,1.85)--(25.75,1.85);
\draw (26.75,2.4)--(25.75,2.975);
\draw (26.75,2.4)--(26.75,1.85);

\filldraw[lightgray] (6.25,1.85) .. controls (15,5.6) .. (15.5,-0.9)--(16.5,0.1) .. controls (17,5.6) .. (25.75,1.85)--(25.75,2.975) .. controls (16,7) .. (6.25,2.975)--(6.25,1.85);
\draw (6.25,1.85) .. controls (15,5.6) .. (15.5,-0.9);
\draw (6.25,2.975) .. controls (16,7) .. (25.75,2.975);
\draw (25.75,1.85) .. controls (17,5.6) .. (16.5,0.1);

%Multiplicities
\node at (19.5,11) {$1$};
\node at (12.5,11) {$1$};

\node at (16,7) {$2$};
\node at (14.5,-3.2) {$1$};
\node at (16,4.5) {$1$};

\node at (9.175,-1.2) {$1$};
\node at (7.175,-0.05) {$1$};
\node at (5.5,1) {$1$};

\node at (22.825,-1.2) {$1$};
\node at (24.825,-0.05) {$1$};
\node at (26.5,1) {$1$};
\end{tikzpicture}
\caption{Fibre of type $\mathrm{I}_0^*$.}
\label{fig:I0*}
  \end{center}
\end{figure}

Finally, one may check that if one pulls back $\overline{\calX}$ by a double cover $\delta \mapsto \delta^2$, the resulting threefold is non-normal along the locus $\{x = \delta= 0\}$. After normalizing, the two components $\{y = \delta = 0\}$ and $\{z = \delta = 0\}$ become exceptional and may be contracted. Resolving singularities of the resulting threefold gives a family with central fibre a smooth K3, as expected. 

\subsection{Fibres with \texorpdfstring{$\gamma = -1$}{gamma = -1}} These fibres turn out to be mild degenerations of fibres of types $\mathrm{I}_0$ and $\mathrm{I}_0^*$, and so are quite simple to describe. The reader may wish to compare the results in this section with \cite[Proposition 3.5]{cytfhrlpk3s}, which gives an analogous description of singular fibres for threefolds fibred by $M_n$-polarized K3 surfaces with $n>1$.

\subsubsection{Types \texorpdfstring{$\mathrm{III}$ and $\mathrm{I}_0$}{III and I0}} Data: $\gamma = -1$, $d \geq 1$, $\det(\rho) = (-1)^d$.

These data are realized by the family $\overline{\calX} \to \Delta$ given by restricting Equation \eqref{eq:X1} to a small disc $\Delta$ containing $\gamma = -1$ and its pull-back under $d$-fold covers. 

\begin{proposition} \label{prop:III} In a neighbourhood of $\gamma = -1$, after crepantly resolving the curves $C_1^{A_7}$, $C_2^{A_7}$, $C^{A_3}$, and $C^{A_1}$, the threefold $\overline{\calX}$ becomes smooth and its singular fibre over $\gamma = -1$ contains an isolated $A_1$ singularity.  

If we pull-back $\overline{\calX}$ by a $d$-fold cover totally ramified over $\gamma = -1$ and crepantly resolve the curves $C_1^{A_7}$, $C_2^{A_7}$, $C^{A_3}$, and $C^{A_1}$, then the singular fibre will still contain an isolated $A_1$ singularity and the resulting threefold will contain an isolated terminal singularity that is locally analytically isomorphic to 
\[0 \in \{s^2 + t^2 + u^2 + v^{d} = 0\} \subset \bA^4[s,t,u,v].\] 
The generalized homological invariant evaluated on a simple closed loop around the preimage of $\gamma = -1$ is conjugate to
\[ \begin{pmatrix} (-1)^d & 0 & 0 \\ 0 & 1 & 0 \\ 0 & 0 & 1 \end{pmatrix}.\]
This fibre type is called $\mathrm{III}$ if $d$ is odd and $\mathrm{I}_0$ if $d$ is even.
\end{proposition}
\begin{proof} This is an explicit computation. The statement about the generalized homological invariant of the resolution of $\overline{\calX}$ follows from Lemma \ref{lem:X1monodromy}, and its $d$-fold cover is obtained by raising this matrix to the $d$th power. \end{proof}

\begin{remark} \label{rem:isolatedsingularity} When $d$ is even, one may perform a small \emph{analytic} resolution of the isolated threefold singularity to obtain a smooth K3 fibre sitting inside a smooth threefold. This is reflected in our naming convention, which combines this case with the smooth $\mathrm{I}_0$ case. Whether or not such a resolution exists \emph{algebraically} depends upon the global geometry of the threefold in which the fibre appears; this is rather a subtle question which we will not address here. 
\end{remark}

\subsubsection{Types \texorpdfstring{$\mathrm{III}^*$ and $\mathrm{I}_0^*$}{III* and I0*}} Data: $\gamma = -1$, $d \geq 1$, $\det(\rho) = (-1)^{d+1}$.

These data are realized by Equation \eqref{eq:2dimX1} when $\beta'$ has a simple pole and $\gamma$ takes the value $-1$ with order $d$. To construct this, for $\delta \in \Delta$, we set $\beta'(\delta) = \frac{1}{\delta}$ and $\gamma =  (\delta^d -1)$ and clear denominators to give
\begin{equation} \label{eq:III*}\frac{\delta(\delta^d -1)}{27} w^4 +\delta w^2yz + \delta wy^2z + \delta wyz^2 + x^2yz = 0,\end{equation}
which defines a family $\overline{\calX} \to \Delta$. The fibre of this family over $\delta = 0$ is almost identical to the central fibre in the $\mathrm{I}_0^*$ family \eqref{eq:I0*}: the only difference is that the dotted line of $A_1$ singularities in Figure \ref{fig:I0*unresolved}  is now a nodal cubic in the component $\{x = 0\}$, defined by the equation $- \frac{1}{27}w^3 + wyz + y^2z + yz^2 = 0$. The node in this curve appears at the point $(w,y,z) = (-3,1,1)$, which is away from its intersections with the other special curves $\{w = 0\}$, $\{y = 0\}$, and $\{z = 0\}$.

\begin{proposition} \label{prop:III*} The threefold defined by Equation \eqref{eq:III*} admits a partial crepant resolution to a threefold with an isolated terminal singularity. This singularity appears above the node in the cubic curve of $A_1$ singularities in the fibre over $\delta = 0$ and is locally analytically isomorphic to
\[0 \in \{s^2 + t^2 + v(u^2 + v^{d-1}) = 0\} \subset \bA^4[s,t,u,v].\] 
Note that if $d = 1$, this singularity splits into a pair of isolated threefold nodes. The fibre of this threefold over $\delta = 0$ contains $11$ components which meet normally away from the singularity. The generalized homological invariant evaluated on a simple closed loop around $\delta = 0$ is conjugate to
\[\begin{pmatrix} (-1)^{d+1} & 0 & 0 \\ 0 & -1 & 0 \\ 0 & 0 & -1 \end{pmatrix}.\]
This fibre type is called $\mathrm{III}^*$ if $d$ is odd and $\mathrm{I}_0^*$ if $d$ is even.
\end{proposition}

\begin{proof} The resolution here proceeds analogously to the $\mathrm{I}_0^*$ case (Subsection \ref{sec:I0*}), with only minor adjustments to deal with the singular cubic curve; we therefore don't reproduce it here. The form of the generalized homological invariant follows from Proposition \ref{prop:generalmonodromy}.
\end{proof}

\subsection{Fibres with \texorpdfstring{$\gamma = 0$}{gamma = 0}} The reader may wish to compare the results in this section with \cite[Proposition 3.7]{cytfhrlpk3s}, which gives an analogous description of singular fibres for  threefolds fibred by $M_n$-polarized K3 surfaces with $n>1$.

\subsubsection{Type \texorpdfstring{$\mathrm{I}_d$}{Id}} Data: $\gamma = 0$, $d \geq 1$, $\det(\rho) = 1$. 

These data are realized by the family $\overline{\calX} \to \Delta$ given by restricting Equation \eqref{eq:X1} to a small disc $\Delta$ containing $\gamma = 0$ and its pull-back under $d$-fold covers. 

\begin{proposition} \label{prop:Id} In a neighbourhood of $\gamma = 0$, the threefold $\overline{\calX}$ is only singular along the four curves $C_1^{A_7}$, $C_2^{A_7}$, $C^{A_3}$, and $C^{A_1}$. After blowing up each curve $i$ times we are left with three isolated threefold nodes, which admit a small resolution to a smooth threefold with trivial canonical sheaf. The resulting smooth threefold has a semistable  singular fibre with three components over $\gamma = 0$.

If we pull back by an $d$-fold cover totally  ramified over $\gamma = 0$, then we may crepantly resolve all singularities to obtain a smooth threefold containing a semi\-stable fibre with $d^2 + 2$ components. The generalized homological invariant evaluated on a simple closed loop around the preimage of $\gamma = 0$ is conjugate to
\[\begin{pmatrix} 1 & d & \tfrac{d(d-1)}{2} \\ 0 & 1 & d \\ 0 & 0 & 1 \end{pmatrix}.\]
This fibre type is called $\mathrm{I}_d$.
\end{proposition}

\begin{proof} A  version of this result, using a different birational model for the threefold, follows from a result of Przyjalkowski \cite[Corollary 4.9]{cyctlgmsft}. We include a proof here for completeness.

The statement that the threefold $\overline{\calX}$ is only singular along the four curves $C_1^{A_7}$, $C_2^{A_7}$, $C^{A_3}$, and $C^{A_1}$ is an explicit computation. The fibre over $\gamma = 0$ has three components $\{y=0\}$, $\{z=0\}$, and $\{w^2 + wy + wz + x^2 = 0\}$. The curve $C^{A_1}$ intersects the interior of the component  $\{w^2 + wy + wz + x^2 = 0\}$ in an $A_1$ singularity, so this curve may be crepantly resolved by a single blow-up. The curve $C^{A_3}$ intersects the double curve $\{y=z=0\}$ and the singularity along this curve jumps to a $A_4$ at $\gamma = 0$; blowing this curve up three times leaves an isolated singularity at its intersection with $\gamma = 0$. Similarly, the curves $C^{A_7}_i$ intersect the two double curves on the component  $\{w^2 + wy + wz + x^2 = 0\}$ and the singularities along these curves jump to $A_8$'s at $\gamma = 0$; blowing these curves up seven times leave isolated singularities at their intersections with $\gamma = 0$. This leaves three isolated threefold nodes, each of which lies on a double curve; these admit small resolutions given by blowing up along components.

The statement about the covers follows immediately from \cite[Proposition 1.2]{bgd7}, and their generalized homological invariants are given by the $d$th power of the generalized homological invariant for ${\calX}_1$, as computed by Lemma \ref{lem:X1monodromy}.
\end{proof}

\subsubsection{Type \texorpdfstring{$\mathrm{I}_d^*$}{Id*}} Data: $\gamma = 0$, $d \geq 1$, $\det(\rho) = -1$. 

These data are realized by Equation \eqref{eq:2dimX1} when $\beta'$ has a simple pole and $\gamma$ vanishes to order $d$. To construct this, for $\delta \in \Delta$, we set $\beta'(\delta) = \frac{1}{\delta}$ and $\gamma =  \delta^d$ and clear denominators to give
\begin{equation} \frac{\delta^{d+1}}{27}w^4 + \delta w^2yz + \delta wy^2z + \delta wyz^2 + x^2yz = 0.\label{eq:Id*},\end{equation}
which defines a family $\overline{\calX} \to \Delta$. The central fibre $\delta = 0$ is shown in Figure \ref{fig:Id*unresolved}.

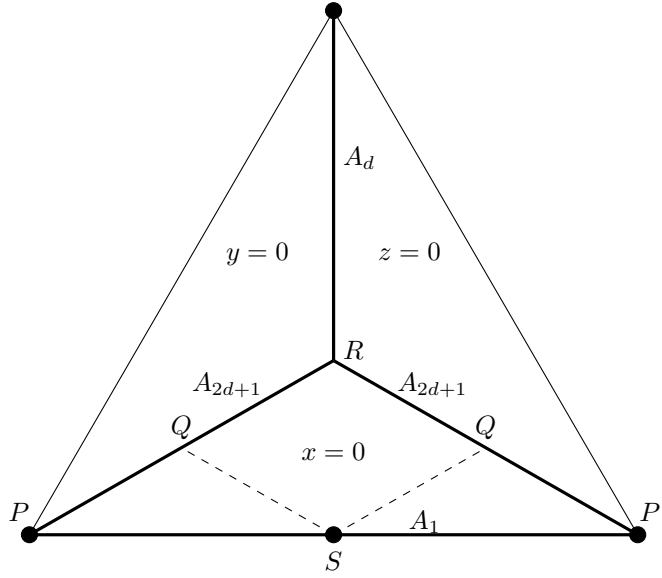
\begin{figure}
\begin{center}
\begin{tikzpicture}
\draw[very thick] (0,0)--(8,0);
\draw (0,0)--(4,6.95)--(8,0);
\draw[very thick] (4,6.95)--(4,2.31);
\draw[very thick] (0,0)--(4,2.31);
\draw[very thick] (4,2.31)--(8,0);
\draw [dashed] (4,0) -- (2,1.15);
\draw [dashed] (4,0) -- (6,1.15);

\filldraw[black]
(0,0) circle (3pt)
(4,0) circle (3pt)
(8,0) circle (3pt)
(4,6.95) circle (3pt)
;

\node[above] at (5.2,-0.1) {$A_1$};
\node[right] at (7.9,0.3) {$P$};
\node[right] at (-0.4,0.3) {$P$};
\node[above] at (5.3,1.7) {$A_{2d+1}$};
\node[above] at (2.6,1.7) {$A_{2d+1}$};
\node[right] at (4,5) {$A_d$};
\node[below] at (3,4) {$y=0$};
\node[below] at (5,4) {$z=0$};
\node at (4,1.1) {$x = 0$};
\node[above] at (2,1.15) {$Q$};
\node[above] at (6,1.15) {$Q$};
\node[above right] at (4,2.2) {$R$};
\node[below] at (4,-0.1) {$S$};
\end{tikzpicture}
\end{center}
\caption{The fibre $\delta = 0$ in Equation \eqref{eq:Id*}.}
\label{fig:Id*unresolved}
\end{figure}

\begin{proposition} The threefold defined by Equation \eqref{eq:Id*} admits a crepant resolution to a smooth threefold. The central fibre of this threefold  contains $2d^2 + 9d + 10$ components which meet normally. The generalized homological invariant evaluated on a simple closed loop around $\delta = 0$ is conjugate to
\[\begin{pmatrix} -1 & -d &- \tfrac{d(d-1)}{2} \\ 0 & -1 & -d \\ 0 & 0 & -1 \end{pmatrix}.\]
This fibre type is called $\mathrm{I}_d^*$
\end{proposition}

\begin{proof}Away from the points $P$, $Q$ and $R$, and $S$, the resolution is given by blowing up the curves of singularities. We deal with each of these points in turn. 

The point $S$ is a transverse intersection of three curves of $A_1$ singularities, which may be simultaneously resolved by the method of \cite[Section 1]{bgd7}; this method is detailed in the proof of Proposition \ref{prop:I0*}. 

The points $P$ are singularities of type $(d+1,4)$ and the points $Q$ are singularities of type $(d+1,0)$, which we may resolve by Lemma \ref{lem:singularity}. The components that pass through them are the hyperplanes from this lemma, so the resulting configurations of divisors meet normally. Each point $P$ gives rise to $(d+3)$ new components and each point $Q$ gives rise to $(d-1)$ new components, in addition to those coming from the curves of $A_k$ singularities. 

The point $R$ is much more difficult: it is an intersection between curves of types $A_{2d+1}$, $A_{2d+1}$, and $A_d$, with local form
\begin{equation} \label{eq:R} \{stu^2 + stv + v^{d+1} = 0\} \subset \bA^4[s,t,u,v].\end{equation} 
The strategy here is to blow-up the vertical curve of type $A_d$ from Figure \ref{fig:Id*unresolved}. This will give rise to a number of new components and a new singularity of the same type. We may then iterate the procedure inductively to resolve all of the singularities.

We first need a couple of base cases: $d=1$ and $d=2$. After blowing up the curve of type $A_d$ we obtain the configurations shown in Figures \ref{fig:Rforn=1} and \ref{fig:Rforn=2}.

In the case $d=1$, after blowing up the vertical curve of type $A_1$ from Figure \ref{fig:Id*unresolved} we are left with a horizontal curve of type $A_1$ and two curves of types $A_3$, which are the strict transforms of the curves of types $A_3$ from before the blow-up, as shown in Figure \ref{fig:Rforn=1}. These curves meet in two points, marked $R'$ in Figure \ref{fig:Rforn=1}, which are singularities of type $(2,0)$; these may be resolved by Lemma \ref{lem:singularity}. Two of the components meeting at these points are the hyperplanes from Lemma \ref{lem:singularity} and, in the notation of that lemma, the third is the hypersurface $\{t = s^2 + uv^2 = 0\}$, so the resulting configuration of divisors meets normally. This lemma shows that the points $R'$ do not give rise to any additional exceptional components, so we only obtain $1$ additional component from the point $R$ in this case, given by blowing up the horizontal $A_1$ curve from Figure \ref{fig:Rforn=1}. The resolved fibre is illustrated in Figure \ref{fig:I1*}.

\begin{figure}
\begin{center}
\begin{tikzpicture}[scale=0.9]
\draw [dotted] (0,0)--(12,0);
\draw [dotted] (0,0)--(4,6.95)--(8,6.95)--(12,0);
\draw (4,6.95)--(4,2.31);
\draw (8,6.95)--(8,2.31);
\draw[very thick] (4,2.31)--(8,2.31);
\draw[very thick] (0,0)--(4,2.31);
\draw[very thick] (8,2.31)--(12,0);

\node[above] at (9.1,1.7) {$A_{3}$};
\node[above] at (2.8,1.7) {$A_{3}$};
\node[above] at (6,1.7) {$A_1$};
\node[above left] at (4.1,2.2) {$R'$};
\node[above right] at (7.9,2.2) {$R'$};
\node[below] at (3,4) {$y=0$};
\node[below] at (9,4) {$z=0$};
\node at (6,1.1) {$x = 0$};
\end{tikzpicture}
\end{center}
\caption{Blow-up of point of type $R$ with $d=1$.}
\label{fig:Rforn=1}
\end{figure}

\begin{figure}
\begin{center}
\begin{tikzpicture}[scale=0.9]
\draw [dotted] (0,0)--(12,0);
\draw [dotted] (0,0)--(4,6.95)--(8,6.95)--(12,0);
\draw (4,6.95)--(4,2.31);
\draw (8,6.95)--(8,2.31);
\draw [very thick] (4,2.31)--(8,2.31);
\draw [very thick] (0,0)--(4,2.31);
\draw (4,2.31)--(6,5.77);
\draw (6,5.77)--(8,2.31);
\draw [very thick] (8,2.31)--(12,0);
\draw (6,5.77)--(6,6.95);

\draw[dashed] (4,2.31) .. controls (6,3.5) .. (8,2.31);

\node[above] at (9.1,1.7) {$A_{5}$};
\node[above] at (2.8,1.7) {$A_{5}$};
\node[above] at (6,1.7) {$A_1$};
\node[below] at (3,4) {$y=0$};
\node[below] at (9,4) {$z=0$};
\node at (6,1.1) {$x = 0$};
\node[above left] at (4.1,2.2) {$R'$};
\node[above right] at (7.9,2.2) {$R'$};
\end{tikzpicture}
\end{center}
\caption{Blow-up of point of type $R$ with $d=2$.}
\label{fig:Rforn=2}
\end{figure}

\begin{figure}
  \begin{center}
\begin{tikzpicture}[scale=0.33]
%Central "triangles"
\draw (27.5,3.15)--(18.6,18.5)--(17.7,18)--(16.8,18.5)--(15.2,18.5)--(14.3,18)--(13.4,18.5)--(4.5,3.15);

%Vertical lines at the top
\draw (17.7,8.2)--(17.7,18);
\draw (14.3,8.2)--(14.3,18);
\draw [lightgray] (18.6,18.5)--(18.6,20);
\draw [lightgray] (16.8,18.5)--(16.8,20);
\draw [lightgray] (15.2,18.5)--(15.2,20);
\draw [lightgray] (13.4,18.5)--(13.4,20);

%Bottom components
\draw (12.5,-2.05)--(19.5,-2.05);
\draw (14.5,-5.5)--(17.5,-5.5);

%Middle components
\draw (14.3,8.2)--(15.2,6.65)--(14.3,5.1)--(15.2,3.35);
\draw (17.7,8.2)--(16.8,6.65)--(17.7,5.1)--(16.8,3.35);
\draw (15.2,6.65)--(16.8,6.65);
\draw (15.2,3.35)--(16.8,3.35);

%left corner
\draw (3.5,2)--(4.5,3.15)--(5.5,3.15)--(14.3,8.2);
\draw (5.5,3.15)--(6.5,1.4);
\draw (3.5,2)--(4.5,0.25);

\draw (4.5,0.25)--(5.85,0.25)--(6.5,1.4)--(7.85,1.4)--(14.3,5.1);
\draw (7.85,1.4)--(8.5,0.25);
\draw (5.85,0.25)--(6.5,-0.9);

\draw (6.5,-0.9)--(7.85,-0.9)--(8.5,0.25)--(9.85,0.25)--(15.2,3.35);
\draw (9.85,0.25)--(10.5,-0.9);
\draw (7.85,-0.9)--(8.5,-2.05);

\draw (8.5,-2.05)--(9.85,-2.05)--(10.5,-0.9)--(12.5,-2.05);
\draw (9.85,-2.05)--(10.5,-3.2);

\draw (10.5,-3.2)--(11.85,-3.2)--(12.5,-2.05);
\draw (11.85,-3.2)--(12.5,-4.35);

\draw (12.5,-4.35)--(14.5,-5.5);

\draw [lightgray] (1.5,0.85)--(3.5,2);
\draw [lightgray] (2.5,-0.9)--(4.5,0.25);
\draw [lightgray] (4.5,-2.05)--(6.5,-0.9);
\draw [lightgray] (6.5,-3.2)--(8.5,-2.05);
\draw [lightgray] (8.5,-4.35)--(10.5,-3.2);
\draw [lightgray] (10.5,-5.5)--(12.5,-4.35);
\draw [lightgray] (12.5,-6.65)--(14.5,-5.5);

%right corner
\draw (28.5,2)--(27.5,3.15)--(26.5,3.15)--(17.7,8.2);
\draw (26.5,3.15)--(25.5,1.4);
\draw (28.5,2)--(27.5,0.25);

\draw (27.5,0.25)--(26.15,0.25)--(25.5,1.4)--(24.15,1.4)--(17.7,5.1);
\draw (24.15,1.4)--(23.5,0.25);
\draw (26.15,0.25)--(25.5,-0.9);

\draw (25.5,-0.9)--(24.15,-0.9)--(23.5,0.25)--(22.15,0.25)--(16.8,3.35);
\draw (22.15,0.25)--(21.5,-0.9);
\draw (24.15,-0.9)--(23.5,-2.05);

\draw (23.5,-2.05)--(22.15,-2.05)--(21.5,-0.9)--(19.5,-2.05);
\draw (22.15,-2.05)--(21.5,-3.2);

\draw (21.5,-3.2)--(20.15,-3.2)--(19.5,-2.05);
\draw (20.15,-3.2)--(19.5,-4.35);

\draw (19.5,-4.35)--(17.5,-5.5);

\draw [lightgray] (30.5,0.85)--(28.5,2);
\draw [lightgray] (29.5,-0.9)--(27.5,0.25);
\draw [lightgray] (27.5,-2.05)--(25.5,-0.9);
\draw [lightgray] (25.5,-3.2)--(23.5,-2.05);
\draw [lightgray] (23.5,-4.35)--(21.5,-3.2);
\draw [lightgray] (21.5,-5.5)--(19.5,-4.35);
\draw [lightgray] (19.5,-6.65)--(17.5,-5.5);

%Raised components
\filldraw [lightgray] (11.5,6)--(11.5,4.9)--(10.75,6.2)--(10.75,7.3);
\draw (11.5,4.9)--(10.75,6.2)--(10.75,7.3);

\filldraw[white] (1.75,0.4)--(5.25,2.4)--(5.25,1.85)--(4.25,0.7)--(2.25,-0.45)--(1.75,0.4);
\filldraw[lightgray] (5.25,1.85)--(6.25,1.85)--(12.5,5.45) .. controls (13.5,6) .. (14.5,5.45)--(15.6,6)--(15,6.3) .. controls (13.5,7.15) .. (12,6.3) --(5.25,2.4)--(5.25,1.85);
\draw [lightgray] (5.25,2.4)--(1.75,0.4);
\draw [lightgray] (4.25,0.7)--(2.25,-0.45);
\draw (4.25,0.7)--(5.25,1.85)--(6.25,1.85)--(12.5,5.45);
\draw (5.25,2.4)--(12,6.3);
\draw (5.25,2.4)--(5.25,1.85);
\draw (15,6.3)--(15.6,6);
\draw (14.5,5.45)--(15.6,6);
\draw (12,6.3) .. controls (13.5,7.15) .. (15,6.3);
\draw (12.5,5.45) .. controls (13.5,6) .. (14.5,5.45);

\filldraw [lightgray] (20.5,6)--(20.5,4.9)--(21.25,6.2)--(21.25,7.3);
\draw (20.5,4.9)--(21.25,6.2)--(21.25,7.3);

\filldraw[white] (30.15,0.4)--(26.75,2.4)--(26.75,1.85)--(27.75,0.7)--(29.75,-0.45)--(30.25,0.4);
\filldraw[lightgray] (26.75,1.85)--(25.75,1.85)--(19.5,5.45) .. controls (18.5,6) .. (17.5,5.45)--(16.4,6)--(17,6.3) .. controls (18.5,7.15) .. (20,6.3) --(26.75,2.4)--(26.75,1.85);
\draw [lightgray] (26.75,2.4)--(30.25,0.4);
\draw [lightgray] (27.75,0.7)--(29.75,-0.45);
\draw (27.75,0.7)--(26.75,1.85)--(25.75,1.85)--(19.5,5.45);
\draw (26.75,2.4)--(20,6.3);
\draw (26.75,2.4)--(26.75,1.85);
\draw (17,6.3)--(16.4,6);
\draw (17.5,5.45)--(16.4,6);
\draw (20,6.3) .. controls (18.5,7.15) .. (17,6.3);
\draw (19.5,5.45) .. controls (18.5,6) .. (17.5,5.45);

\filldraw[white] (16.5,-7)--(16.5,-3.275)--(16,-3.775)--(15.5,-5.5)--(15.5,-7)--(16.5,-7);
\filldraw[lightgray] (16,-3.775)--(15.5,-2.05)--(16.5,-2.05)--(16.5,-3.275)--(16,-3.775);
\draw [lightgray] (15.5,-7)--(15.5,-5.5);
\draw [lightgray] (16.5,-7)--(16.5,-3.275);
\draw (15.5,-5.5)--(16,-3.775)--(15.5,-2.05);
\draw (16.5,-3.275)--(16.5,-2.05);
\draw (16,-3.775)--(16.5,-3.275);

\filldraw [lightgray] (15.5,-2.05)--(11.5,4.9)--(11.5,6)--(15.5,-0.95);
\filldraw [lightgray] (16.5,-2.05)--(20.5,4.9)--(20.5,6)--(16.5,-0.95);
\filldraw [lightgray] (15.5,-2.05)--(13.6,1.25)--(13.6,2.35) .. controls (16,-1.8) .. (18.4,2.35)--(18.4,1.25)--(16.5,-2.05);
\draw (15.5,-2.05)--(11.5,4.9)--(11.5,6);
\draw (10.75,7.3)--(13.6,2.35);
\draw (16.5,-2.05)--(20.5,4.9)--(20.5,6);
\draw (21.25,7.3)--(18.4,2.35);
\draw (13.6,2.35) .. controls (16,-1.8) .. (18.4,2.35);

%Multiplicities
\node at (16,10) {$1$};
\node at (20,10) {$1$};
\node at (12,10) {$1$};

\node at (16,4.8) {$2$};
\node at (16,1) {$2$};
\node at (14.5,-3.775) {$1$};
\node at (16,-1.5) {$1$};

\node at (11,2) {$2$};
\node at (9.175,-0.9) {$1$};
\node at (7.175,0.25) {$1$};
\node at (10,3.3) {$2$};
\node at (5.5,1) {$1$};
\node at (11.175,-2.35) {$1$};
\node at (9.3,4.15) {$1$};

\node at (21,2) {$2$};
\node at (22.825,-0.9) {$1$};
\node at (24.825,0.25) {$1$};
\node at (22,3.3) {$2$};
\node at (26.5,1) {$1$};
\node at (20.825,-2.35) {$1$};
\node at (22.7,4.15) {$1$};
\end{tikzpicture}
\caption{Fibre of type $\mathrm{I}_1^*$.}
\label{fig:I1*}
  \end{center}
\end{figure}

In the case $d =2$, after blowing up the vertical curve of $A_2$ singularities from Figure \ref{fig:Id*unresolved} we are left with the configuration shown in Figure \ref{fig:Rforn=2}. The points labelled $R'$ are singularities of type $(3,1)$, which we may resolve by Lemma \ref{lem:singularity} to give $2$ new components from each. The four components that meet at them are the four hyperplanes from this lemma, so the resulting configuration of divisors meets normally. All of the remaining curves resolve in the expected way. We thus obtain $7$ additional components from the point $R$ in this case (the central triangle, two from the curves of $A_1$'s, and two from each of the two points $R'$).

Now we look at the general case $d > 2$. After blowing up the vertical curve of $A_d$ singularities  from Figure \ref{fig:Id*unresolved}, we are left with the configuration shown in Figure \ref{fig:generalR}. In this picture, the points labelled $R'$ are singularities of type $(d+1,d-1)$, which we may resolve by Lemma \ref{lem:singularity} to give $2d-2$ new components from each. The four components that meet at them are the four hyperplanes from this lemma, so the resulting configuration of divisors meets normally. The point $R''$ is an intersection between curves of types $A_{2d-3}$, $A_{2d-3}$, and $A_{d-2}$, which has local form
\[\{stu^2 + stv + v^{d-1} = 0\} \subset \bA^4[s,t,u,v].\]
This is the same setting as Equation \eqref{eq:R}, with $d$ reduced by $2$. The complete resolution of $R$  proceeds by induction and we find that $R$ gives rise to $2d^2 - 1$ components in total.

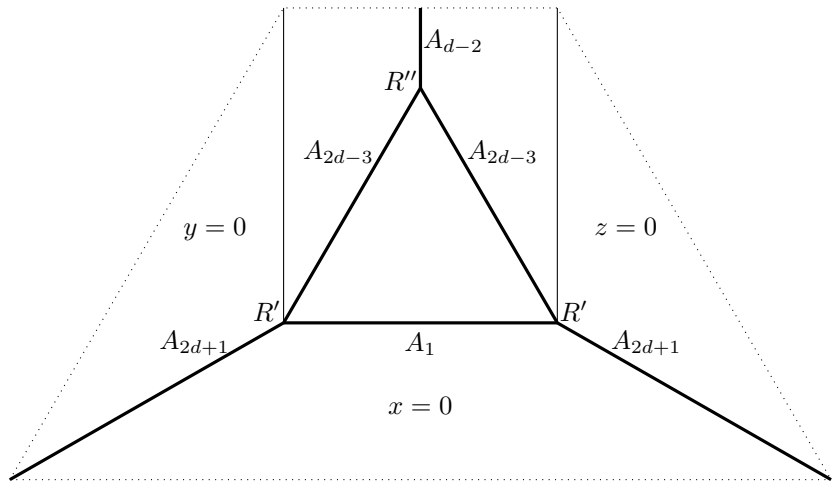
\begin{figure}
\begin{center}
\begin{tikzpicture}[scale=0.9]
\draw [dotted] (0,0)--(12,0);
\draw [dotted] (0,0)--(4,6.95)--(8,6.95)--(12,0);
\draw (4,6.95)--(4,2.31);
\draw  (8,6.95)--(8,2.31);
\draw [very thick] (4,2.31)--(8,2.31);
\draw [very thick] (0,0)--(4,2.31)--(6,5.77);
\draw [very thick] (6,5.77)--(8,2.31)--(12,0);
\draw [very thick] (6,5.77)--(6,6.95);

\node[above] at (9.3,1.7) {$A_{2d+1}$};
\node[above] at (2.7,1.7) {$A_{2d+1}$};
\node[above] at (6,1.7) {$A_1$};
\node[below] at (3,4) {$y=0$};
\node[below] at (9,4) {$z=0$};
\node at (6,1.1) {$x = 0$};
\node[above left] at (6.1,5.6) {$R''$};
\node[above left] at (4.1,2.2) {$R'$};
\node[above right] at (7.9,2.2) {$R'$};
\node[above] at (7.2,4.5) {$A_{2d-3}$};
\node[above] at (4.8,4.5) {$A_{2d-3}$};
\node[right] at (5.9,6.45) {$A_{d-2}$};
\end{tikzpicture}
\end{center}
\caption{Blow-up of point of type $R$ with $d> 2$.}
\label{fig:generalR}
\end{figure}

After all blow-ups, we obtain a smooth threefold with trivial canonical sheaf, whose central fibre is normal crossings and contains:
\begin{itemize} 
\item $3$ components as the strict transforms of the original $3$;
\item $5d + 4$ components from the curves of $A_k$ singularities;
\item $2(d+3)$ components from the points $P$;
\item $2(d-1)$ components from the points $Q$;
\item $2d^2 -1$ components from the point $R$.
\end{itemize}
Summing up, we obtain $2d^2 + 9n + 10$ components. The form of the generalized homological invariant evaluated around a closed loop follows from Proposition \ref{prop:generalmonodromy}.\end{proof}

Finally, it may be checked that after pulling a fibre of type $\mathrm{I}_d^*$ back by a double cover ramified at $\delta = 0$, normalizing, and contracting exceptional components, one obtains a fibre of type $\mathrm{I}_{2d}$, as expected.

\subsection{Fibres with \texorpdfstring{$\gamma = \infty$}{gamma = infinity}}\label{subsec:gammainfinity}

The most difficult fibres occur over $\gamma = \infty$. The reader may wish to compare the results in this section with \cite[Proposition 3.8]{cytfhrlpk3s}, which gives an analogous description of singular fibres for  threefolds fibred by $M_n$-polarized K3 surfaces with $n>1$.

\subsubsection{Types \texorpdfstring{$\mathrm{II}^*$, $\mathrm{IV}^*$, $\mathrm{I}_0^*$}{II*, IV*, I0*}} Data: $\gamma = \infty$, $d \in \{1,2,3\}$, $\det(\rho) = (-1)^d$. 

These data are most easily realized by the family $\overline{\calX} \to \Delta$ given by restricting Equation \eqref{eq:X1} to a small disc $\Delta$ containing $\gamma = \infty$ and its pull-back under $d$-fold covers. However, we instead consider a more general family $\overline{\calX} \to \Delta$ defined by
\begin{equation} \label{eq:infinity1} \{\tfrac{1}{27}w^4+ \delta^{d+k}w^2yz + \delta^dwy^2z + \delta^dwyz^2 + \delta^dx^2yz   = 0\} \subset \bP^3[w,x,y,z] \times \Delta,\end{equation}
where $k \geq 0$ and $d \in \{1,2,3\}$. Note that if $k = 0$ this family is obtained from Equation \eqref{eq:X1} by setting $\gamma = \frac{1}{\delta^d}$, then clearing denominators, but if $k > 0$ it is not in one of our standard forms. The choice of a family that is not in Weierstrass form may seem unexpected, but we will require the extra generality in order to complete an induction argument for $d > 3$ in Section \ref{sec:modularity}.

\begin{proposition} \label{prop:infinityfibres1}  Consider the threefold $\overline{\pi}\colon\overline{\calX} \to \Delta$ defined by Equation \eqref{eq:infinity1}. For any $k \geq 0$, after resolving the singularities of $\overline{\calX}$ we may perform a birational modification, affecting only the fibre over $\delta = 0$, so that the resulting threefold is smooth and has trivial canonical sheaf. In the cases where $d = 1$, $2$, or $3$ the resulting fibre contains $53$, $22$, or $11$ components, respectively, and these components meet normally. The generalized homological invariant evaluated on a simple closed loop around $\delta = 0$ is conjugate to
\[\begin{pmatrix}(-1)^d & 0 & 0 \\ 0 & \omega^d & 0 \\ 0 & 0 & \omega^{5d} \end{pmatrix},\]
where $\omega$ is a primitive sixth root of unity. For $d = 1$, $2$, and $3$ the resulting fibre type is called $\mathrm{II}^*$, $\mathrm{IV}^*$, and $\mathrm{I}_0^*$ respectively.
\end{proposition}

\begin{remark} We note that descriptions of the singular fibres for $d=1,2$ have previously appeared in work of Przyjalkowski \cite[Theorem 22, cases $X_2$ and $V_1$]{wlgmsft}, albeit using a somewhat different family over $\Delta$, and explicit computations of the fibre in the case $d=1$ have appeared in work of Iliev, Katzarkov, and Przyjalkowski \cite[Section 5.3]{dscnr}.
\end{remark}

\begin{proof}[Proof of Proposition \ref{prop:infinityfibres1}]  We begin with the case $d=1$. After blowing up the fibre over  $\delta = 0$ four times in the fourfold ambient space we obtain the configuration shown in Figure \ref{fig:d=1}.  

\begin{figure}
  \begin{center}
\begin{tikzpicture}[scale=0.7]
\draw [very thick] (0,0)--(16,0);
\draw (0,0)--(8,13.85)--(16,0);
\draw [very thick] (2,1.15)--(14,1.15);
\draw (14,1.15)--(8,11.55)--(2,1.15);
\draw [very thick] (4,2.3)--(12,2.3);
\draw (12,2.3)--(8,9.25)--(4,2.3);
\draw (6,3.45)--(10,3.45)--(8,6.95)--(6,3.45);
\draw [very thick] (0,0)-- (4,2.3);
\draw (4,2.3)--(6,3.45);
\draw [dashed] (8,0)--(8,2.3);
\draw [very thick] (8,13.85)--(8,9.23);
\draw (8,9.25)--(8,6.95);
\draw (10,3.45)--(12,2.3);
\draw [very thick] (12,2.3)--(16,0);
\draw [dashed] (8,2.3) .. controls (9,3.3) .. (12,2.3);
\draw [dashed] (4,2.3) .. controls (7,3.3) .. (8,2.3);

\filldraw[black]
(0,0) circle (3pt)
(16,0) circle (3pt)
(8,0) circle (3pt)
(8,13.85) circle (3pt)
;

\node[above] at (9.5,-0.1) {$A_1$};
\node[above] at (9.5,1.05) {$A_1$};
\node[above] at (9.5,2.2) {$A_1$};
\node[right] at (7.9,9.8) {$A_1$};
\node[right] at (7.9,12.1) {$A_2$};
\node[right] at (1,0.5) {$A_5$};
\node[right] at (14,0.5) {$A_5$};
\node[right] at (3,1.65) {$A_3$};
\node[right] at (12,1.65) {$A_3$};
\node[right] at (11.9,2.55) {$P_1$};
\node[right] at (13.9,1.4) {$P_2$};
\node[right] at (15.8,0.3) {$P_3$};
\node[right] at (3.25,2.55) {$P_1$};
\node[right] at (1.25,1.4) {$P_2$};
\node[right] at (-0.75,0.3) {$P_3$};
\end{tikzpicture}
\caption{Partial resolution of fibre $\overline{\pi}^{-1}(0)$ in case $d = 1$ of Proposition \ref{prop:infinityfibres1}.}
\label{fig:d=1}
  \end{center}
\end{figure}

In this figure, the points $P_i$ are singularities of type $(m,n)$, which we may resolve by Lemma \ref{lem:singularity}. The point $P_1$ has $(m,n) = (2,1)$, $P_2$ has $(m,n) = (3,2)$, and $P_3$ has $(m,n) = (4,3)$. The four components that meet at them are the four hyperplanes from this lemma, so the resulting configuration of divisors meets normally.  We obtain one new component from each $P_1$, three new components from each $P_2$, and five new components from each $P_3$. The resulting fibre of type $\mathrm{II}^*$ is shown in Figure \ref{fig:II*}. 

\begin{figure}
  \begin{center}
\begin{tikzpicture}[scale=0.33]
%Central "triangles"
\draw (31.5,0.85)--(19.9,20.8)--(18.6,20.05)--(17.3,20.8)--(16,20.05)--(14.7,20.8)--(13.4,20.05)--(12.1,20.8)--(0.5,0.85);
\draw (27.5,3.15)--(18.6,18.5)--(17.3,17.75)--(16,18.5)--(14.7,17.75)--(13.4,18.5)--(4.5,3.15);
\draw (23.5,5.45)--(17.3,16.2)--(16,15.45)--(14.7,16.2)--(8.5,5.45);
\draw (12,6.9)--(20,6.9)--(16,13.9)--(12,6.9);

%Vertical lines at the top
\draw (16,13.9)--(16,15.45);
\draw (17.3,16.2)--(17.3,17.75);
\draw (14.7,16.2)--(14.7,17.75);
\draw (18.6,18.5)--(18.6,20.05);
\draw (16,18.5)--(16,20.05);
\draw (13.4,18.5)--(13.4,20.05);
\draw [lightgray] (19.9,20.8)--(19.9,22);
\draw [lightgray] (17.3,20.8)--(17.3,22);
\draw [lightgray] (14.7,20.8)--(14.7,22);
\draw [lightgray] (12.1,20.8)--(12.1,22);

%left corner
\draw (-0.5,-0.3)--(0.5,0.85)--(1.5,0.85)--(3.5,2)--(4.5,3.15)--(5.5,3.15)--(7.5,4.3)--(8.5,5.45)--(9.5,5.45)--(12,6.9);
\draw (9.5,5.45)--(10.5,3.7);
\draw (7.5,4.3)--(8.5,2.55);
\draw (5.5,3.15)--(6.5,1.4);
\draw (3.5,2)--(4.5,0.25);
\draw (1.5,0.85)--(2.5,-0.9);
\draw (-0.5,-0.3)--(0.5,-2.05);

\draw (0.5,-2.05)--(1.85,-2.05)--(2.5,-0.9)--(3.85,-0.9)--(4.5,0.25)--(5.85,0.25)--(6.5,1.4)--(7.85,1.4)--(8.5,2.55)--(9.85,2.55)--(10.5,3.7);
\draw (9.85,2.55)--(10.5,1.4);
\draw (7.85,1.4)--(8.5,0.25);
\draw (5.85,0.25)--(6.5,-0.9);
\draw (3.85,-0.9)--(4.5,-2.05);
\draw (1.85,-2.05)--(2.5,-3.2);

\draw (2.5,-3.2)--(3.85,-3.2)--(4.5,-2.05)--(5.85,-2.05)--(6.5,-0.9)--(7.85,-0.9)--(8.5,0.25)--(9.85,0.25)--(10.5,1.4);
\draw (9.85,0.25)--(10.5,-0.9);
\draw (7.85,-0.9)--(8.5,-2.05);
\draw (5.85,-2.05)--(6.5,-3.2);
\draw (3.85,-3.2)--(4.5,-4.35);

\draw (4.5,-4.35)--(5.85,-4.35)--(6.5,-3.2)--(7.85,-3.2)--(8.5,-2.05)--(9.85,-2.05)--(10.5,-0.9);
\draw (9.85,-2.05)--(10.5,-3.2);
\draw (7.85,-3.2)--(8.5,-4.35);
\draw (5.85,-4.35)--(6.5,-5.5);

\draw (6.5,-5.5)--(7.85,-5.5)--(8.5,-4.35)--(9.85,-4.35)--(10.5,-3.2);
\draw (9.85,-4.35)--(10.5,-5.5);
\draw (7.85,-5.5)--(8.5,-6.65);

\draw (8.5,-6.65)--(9.85,-6.65)--(10.5,-5.5);
\draw (9.85,-6.65)--(10.5,-7.8);

\draw [lightgray] (-0.5,-0.3)--(-2.5,-1.45);
\draw [lightgray] (0.5,-2.05)--(-1.5,-3.2);
\draw [lightgray] (2.5,-3.2)--(0.5,-4.35);
\draw [lightgray] (4.5,-4.35)--(2.5,-5.5);
\draw [lightgray] (6.5,-5.5)--(4.5,-6.65);
\draw [lightgray] (8.5,-6.65)--(6.5,-7.8);
\draw [lightgray] (10.5,-7.8)--(8.5,-8.95);

%right corner
\draw (32.5,-0.3)--(31.5,0.85)--(30.5,0.85)--(28.5,2)--(27.5,3.15)--(26.5,3.15)--(24.5,4.3)--(23.5,5.45)--(22.5,5.45)--(20,6.9);
\draw (22.5,5.45)--(21.5,3.7);
\draw (24.5,4.3)--(23.5,2.55);
\draw (26.5,3.15)--(25.5,1.4);
\draw (28.5,2)--(27.5,0.25);
\draw (30.5,0.85)--(29.5,-0.9);
\draw (32.5,-0.3)--(31.5,-2.05);

\draw (31.5,-2.05)--(30.15,-2.05)--(29.5,-0.9)--(28.15,-0.9)--(27.5,0.25)--(26.15,0.25)--(25.5,1.4)--(24.15,1.4)--(23.5,2.55)--(22.15,2.55)--(21.5,3.7);
\draw (22.15,2.55)--(21.5,1.4);
\draw (24.15,1.4)--(23.5,0.25);
\draw (26.15,0.25)--(25.5,-0.9);
\draw (28.15,-0.9)--(27.5,-2.05);
\draw (30.15,-2.05)--(29.5,-3.2);

\draw (29.5,-3.2)--(28.15,-3.2)--(27.5,-2.05)--(26.15,-2.05)--(25.5,-0.9)--(24.15,-0.9)--(23.5,0.25)--(22.15,0.25)--(21.5,1.4);
\draw (22.15,0.25)--(21.5,-0.9);
\draw (24.15,-0.9)--(23.5,-2.05);
\draw (26.15,-2.05)--(25.5,-3.2);
\draw (28.15,-3.2)--(27.5,-4.35);

\draw (27.5,-4.35)--(26.15,-4.35)--(25.5,-3.2)--(24.15,-3.2)--(23.5,-2.05)--(22.15,-2.05)--(21.5,-0.9);
\draw (22.15,-2.05)--(21.5,-3.2);
\draw (24.15,-3.2)--(23.5,-4.35);
\draw (26.15,-4.35)--(25.5,-5.5);

\draw (25.5,-5.5)--(24.15,-5.5)--(23.5,-4.35)--(22.15,-4.35)--(21.5,-3.2);
\draw (22.15,-4.35)--(21.5,-5.5);
\draw (24.15,-5.5)--(23.5,-6.65);

\draw (23.5,-6.65)--(22.15,-6.65)--(21.5,-5.5);
\draw (22.15,-6.65)--(21.5,-7.8);

\draw [lightgray] (32.5,-0.3)--(34.5,-1.45);
\draw [lightgray] (31.5,-2.05)--(33.5,-3.2);
\draw [lightgray] (29.5,-3.2)--(31.5,-4.35);
\draw [lightgray] (27.5,-4.35)--(29.5,-5.5);
\draw [lightgray] (25.5,-5.5)--(27.5,-6.65);
\draw [lightgray] (23.5,-6.65)--(25.5,-7.8);
\draw [lightgray] (21.5,-7.8)--(23.5,-8.95);

%Bottom components
\draw (10.5,3.7)--(21.5,3.7);
\draw (10.5,1.4)--(21.5,1.4);
\draw (10.5,-0.9)--(21.5,-0.9);
\draw (10.5,-3.2)--(21.5,-3.2);
\draw (10.5,-5.5)--(21.5,-5.5);
\draw (10.5,-7.8)--(21.5,-7.8);

%Raised components
%\filldraw[lightgray] (16,3.7) .. controls (17,5.6) .. (21.75,4.15)--(21.75,5.275) .. controls (17.5,7) .. (16.5,4.7)--(16,3.7);
%\draw (16,3.7) .. controls (17,5.6) .. (21.75,4.15);
%\draw (16.5,4.7) .. controls (17.5,7) .. (21.75,5.275);
%\filldraw[lightgray] (10.25,4.15) .. controls (15,5.6) .. (15.5,3.7)--(16.5,4.7) .. controls (15.5,7) .. (10.25,5.275)--(10.25,4.15);
%\draw (10.25,4.15) .. controls (15,5.6) .. (15.5,3.7);
%\draw (10.25,5.275) .. controls (15.5,7) .. (16.5,4.7);

\filldraw[white] (16.5,-9)--(16.5,-6.15)--(16,-6.65)--(15.5,-7.8)--(15.5,-9)--(16.5,-9);
\filldraw[lightgray] (16,-6.65)--(15.5,-5.5)--(15.5,-3.2)--(16,-2.05)--(15.5,-0.9)--(15.5,1.4)--(16,2.55)--(15.5,3.7)--(16.5,4.7)--(16.5,-6.15)--(16,-6.65);
\draw [lightgray] (15.5,-9)--(15.5,-7.8);
\draw [lightgray] (16.5,-9)--(16.5,-6.15);
\draw (15.5,-7.8)--(16,-6.65)--(15.5,-5.5)--(15.5,-3.2)--(16,-2.05)--(15.5,-0.9)--(15.5,1.4)--(16,2.55)--(15.5,3.7);
\draw (16.5,-6.15)--(16.5,4.7);
\draw (16,-6.65)--(16.5,-6.15);
\draw (16,-2.05)--(16.5,-1.55);
\draw (16,2.55)--(16.5,3.05);

\filldraw[white] (1.25,0.1)--(-2.25,-1.9)--(-1.75,-2.75)--(0.25,-1.6)--(1.25,-0.45)--(1.25,0.1);
\filldraw[lightgray] (1.25,-0.45)--(2.25,-0.45)--(4.25,0.7)--(5.25,1.85)--(6.25,1.85)--(8.25,3)--(9.25,4.15)--(10.25,4.15)--(10.25,5.275)--(1.25,0.1)--(1.25,-0.45);
\draw [lightgray] (1.25,0.1)--(-2.25,-1.9);
\draw [lightgray] (0.25,-1.6)--(-1.75,-2.75);
\draw (0.25,-1.6)--(1.25,-0.45)--(2.25,-0.45)--(4.25,0.7)--(5.25,1.85)--(6.25,1.85)--(8.25,3)--(9.25,4.15)--(10.25,4.15);
\draw (1.25,0.1)--(10.25,5.275);
\draw (1.25,0.1)--(1.25,-0.45);
\draw (5.25,2.4)--(5.25,1.85);
\draw (9.25,4.7)--(9.25,4.15);

\filldraw[white] (30.75,0.1)--(34.25,-1.9)--(33.75,-2.75)--(31.75,-1.6)--(30.75,-0.45)--(30.75,0.1);
\filldraw[lightgray] (30.75,-0.45)--(29.75,-0.45)--(27.75,0.7)--(26.75,1.85)--(25.75,1.85)--(23.75,3)--(22.75,4.15)--(21.75,4.15)--(21.75,5.275)--(30.75,0.1)--(30.75,-0.45);
\draw [lightgray] (30.75,0.1)--(34.25,-1.9);
\draw [lightgray] (31.75,-1.6)--(33.75,-2.75);
\draw (31.75,-1.6)--(30.75,-0.45)--(29.75,-0.45)--(27.75,0.7)--(26.75,1.85)--(25.75,1.85)--(23.75,3)--(22.75,4.15)--(21.75,4.15);
\draw (30.75,0.1)--(21.75,5.275);
\draw (30.75,0.1)--(30.75,-0.45);
\draw (26.75,2.4)--(26.75,1.85);
\draw (22.75,4.7)--(22.75,4.15);

\filldraw[lightgray] (10.25,4.15) .. controls (15,5.6) .. (15.5,3.7)--(16.5,3.7) .. controls (17,5.6) .. (21.75,4.15)--(21.75,5.275) .. controls (16,7) .. (10.25,5.275)--(10.25,4.15);
\draw (10.25,4.15) .. controls (15,5.6) .. (15.5,3.7);
\draw (10.25,5.275) .. controls (16,7) .. (21.75,5.275);
\draw (21.75,4.15) .. controls (17,5.6) .. (16.5,3.7);

%Multiplicities
\node at (16,10) {$4$};
\node at (19.5,10) {$3$};
\node at (22.25,10) {$2$};
\node at (25,10) {$1$};
\node at (12.5,10) {$3$};
\node at (9.75,10) {$2$};
\node at (7,10) {$1$};
\node at (16,16.8) {$2$};
\node at (14.7,19.2) {$1$};
\node at (17.3,19.2) {$1$};

\node at (14.5,4.4) {$6$};
\node at (14.5,2.55) {$5$};
\node at (14.5,0.25) {$4$};
\node at (14.5,-2.05) {$3$};
\node at (14.5,-4.35) {$2$};
\node at (14.5,-6.65) {$1$};
\node at (16,-4.35) {$1$};
\node at (16,0.25) {$2$};
\node at (16,5.5) {$3$};

\node at (9.175,1.4) {$4$};
\node at (9.175,-0.9) {$3$};
\node at (9.175,-3.2) {$2$};
\node at (9.175,-5.5) {$1$};
\node at (7.175,0.25) {$3$};
\node at (7.175,-2.05) {$2$};
\node at (7.175,-4.35) {$1$};
\node at (5.175,-0.9) {$2$};
\node at (5.175,-3.2) {$1$};
\node at (3.175,-2.05) {$1$};
\node at (9.5,3.3) {$5$};
\node at (7.5,1.95) {$4$};
\node at (5.5,1) {$3$};
\node at (3.5,-0.35) {$2$};
\node at (1.5,-1.3) {$1$};
\node at (2.95,0.5) {$1$};
\node at (6.9,2.8) {$2$};

\node at (22.825,1.4) {$4$};
\node at (22.825,-0.9) {$3$};
\node at (22.825,-3.2) {$2$};
\node at (22.825,-5.5) {$1$};
\node at (24.825,0.25) {$3$};
\node at (24.825,-2.05) {$2$};
\node at (24.825,-4.35) {$1$};
\node at (26.825,-0.9) {$2$};
\node at (26.825,-3.2) {$1$};
\node at (28.825,-2.05) {$1$};
\node at (22.5,3.3) {$5$};
\node at (24.5,1.95) {$4$};
\node at (26.5,1) {$3$};
\node at (28.5,-0.35) {$2$};
\node at (30.5,-1.3) {$1$};
\node at (29.05,0.5) {$1$};
\node at (25.1,2.8) {$2$};
\end{tikzpicture}
\caption{Fibre of type $\mathrm{II}^*$.}
\label{fig:II*}
  \end{center}
\end{figure}

In the case $d=2$, the threefold $\overline{\calX}$ is non-normal along its fibre over $\delta = 0$. After blowing up this fibre twice in the ambient space we obtain a normal threefold, with  fibre over $\delta = 0$ as shown in Figure \ref{fig:d=2}. The points labelled $P$ are singularities of type $(4,2)$, which we may resolve by Lemma \ref{lem:singularity}. The four components that meet at them are the four hyperplanes from this lemma, so the resulting configuration of divisors meets normally. The points labelled $Q$ are singularities of type $(2,0)$, which do not give rise to any new components on resolution. The resolution proceeds similarly to the $d=1$ case; the result is a normal-crossings fibre of type $\mathrm{IV}^*$, shown in Figure \ref{fig:IV*}.  

\begin{figure}
  \begin{center}
\begin{tikzpicture}[scale=0.6]
\draw [very thick] (0,0)--(8,0);
\draw (0,0)--(4,6.95)--(8,0);
\draw [very thick] (2,1.15)--(6,1.15);
\draw (6,1.15)--(4,4.65)--(2,1.15);
\draw [very thick] (0,0)--(2,1.15);
\draw [dashed] (4,0)--(4,1.15);
\draw [very thick] (4,6.95)--(4,4.65);
\draw [very thick] (6,1.15)--(8,0);

\filldraw[black]
(0,0) circle (3pt)
(4,0) circle (3pt)
(8,0) circle (3pt)
(4,6.95) circle (3pt)
;

\node[above] at (5,-0.15) {$A_1$};
\node[above] at (5,1) {$A_1$};
\node[right] at (3.9,5.2) {$A_1$};
\node[right] at (1,0.5) {$A_3$};
\node[right] at (5.9,0.5) {$A_3$};
\node[right] at (5.8,1.4) {$Q$};
\node[right] at (1.3,1.4) {$Q$};
\node[right] at (7.9,0.3) {$P$};
\node[right] at (-0.7,0.3) {$P$};
\end{tikzpicture}
\caption{Partial resolution of fibre $\overline{\pi}^{-1}(0)$ in case $d = 2$ of Proposition \ref{prop:infinityfibres1}.}
\label{fig:d=2}
\end{center}
\end{figure}

\begin{figure}
  \begin{center}
\begin{tikzpicture}[scale=0.33]
%Central "triangles"
\draw (27.5,3.15)--(18.6,18.5)--(17.3,17.75)--(16.8,18.05)--(15.2,18.05)--(14.7,17.75)--(13.4,18.5)--(4.5,3.15);
\draw (23.5,5.45)--(17.3,16.2)--(14.7,16.2)--(8.5,5.45)--(23.5,5.45);

%Vertical lines at the top
\draw (17.3,16.2)--(17.3,17.75);
\draw (14.7,16.2)--(14.7,17.75);
\draw [lightgray] (18.6,18.5)--(18.6,20);
\draw [lightgray] (16.8,18.05)--(16.8,20);
\draw [lightgray] (15.2,18.05)--(15.2,20);
\draw [lightgray] (13.4,18.5)--(13.4,20);

%Bottom components
\draw (10.5,1.4)--(21.5,1.4);
\draw (12.5,-2.05)--(19.5,-2.05);
\draw (14.5,-5.5)--(17.5,-5.5);

%left corner
\draw (3.5,2)--(4.5,3.15)--(5.5,3.15)--(7.5,4.3)--(8.5,5.45);
\draw (7.5,4.3)--(8.5,2.55);
\draw (5.5,3.15)--(6.5,1.4);
\draw (3.5,2)--(4.5,0.25);

\draw (4.5,0.25)--(5.85,0.25)--(6.5,1.4)--(7.85,1.4)--(8.5,2.55)--(10.5,1.4);
\draw (7.85,1.4)--(8.5,0.25);
\draw (5.85,0.25)--(6.5,-0.9);

\draw (6.5,-0.9)--(7.85,-0.9)--(8.5,0.25)--(9.85,0.25)--(10.5,1.4);
\draw (9.85,0.25)--(10.5,-0.9);
\draw (7.85,-0.9)--(8.5,-2.05);

\draw (8.5,-2.05)--(9.85,-2.05)--(10.5,-0.9)--(12.5,-2.05);
\draw (9.85,-2.05)--(10.5,-3.2);

\draw (10.5,-3.2)--(11.85,-3.2)--(12.5,-2.05);
\draw (11.85,-3.2)--(12.5,-4.35);

\draw (12.5,-4.35)--(14.5,-5.5);

\draw [lightgray] (1.5,0.85)--(3.5,2);
\draw [lightgray] (2.5,-0.9)--(4.5,0.25);
\draw [lightgray] (4.5,-2.05)--(6.5,-0.9);
\draw [lightgray] (6.5,-3.2)--(8.5,-2.05);
\draw [lightgray] (8.5,-4.35)--(10.5,-3.2);
\draw [lightgray] (10.5,-5.5)--(12.5,-4.35);
\draw [lightgray] (12.5,-6.65)--(14.5,-5.5);

%right corner
\draw (28.5,2)--(27.5,3.15)--(26.5,3.15)--(24.5,4.3)--(23.5,5.45);
\draw (24.5,4.3)--(23.5,2.55);
\draw (26.5,3.15)--(25.5,1.4);
\draw (28.5,2)--(27.5,0.25);

\draw (27.5,0.25)--(26.15,0.25)--(25.5,1.4)--(24.15,1.4)--(23.5,2.55)--(21.5,1.4);
\draw (24.15,1.4)--(23.5,0.25);
\draw (26.15,0.25)--(25.5,-0.9);

\draw (25.5,-0.9)--(24.15,-0.9)--(23.5,0.25)--(22.15,0.25)--(21.5,1.4);
\draw (22.15,0.25)--(21.5,-0.9);
\draw (24.15,-0.9)--(23.5,-2.05);

\draw (23.5,-2.05)--(22.15,-2.05)--(21.5,-0.9)--(19.5,-2.05);
\draw (22.15,-2.05)--(21.5,-3.2);

\draw (21.5,-3.2)--(20.15,-3.2)--(19.5,-2.05);
\draw (20.15,-3.2)--(19.5,-4.35);

\draw (19.5,-4.35)--(17.5,-5.5);

\draw [lightgray] (30.5,0.85)--(28.5,2);
\draw [lightgray] (29.5,-0.9)--(27.5,0.25);
\draw [lightgray] (27.5,-2.05)--(25.5,-0.9);
\draw [lightgray] (25.5,-3.2)--(23.5,-2.05);
\draw [lightgray] (23.5,-4.35)--(21.5,-3.2);
\draw [lightgray] (21.5,-5.5)--(19.5,-4.35);
\draw [lightgray] (19.5,-6.65)--(17.5,-5.5);

%Raised components
\filldraw[white] (1.75,0.4)--(5.25,2.4)--(5.25,1.85)--(4.25,0.7)--(2.25,-0.45)--(1.75,0.4);
\filldraw[lightgray] (5.25,1.85)--(6.25,1.85)--(8.25,3)--(9.25,4.7)--(5.25,2.4)--(5.25,1.85);
\draw [lightgray] (5.25,2.4)--(1.75,0.4);
\draw [lightgray] (4.25,0.7)--(2.25,-0.45);
\draw (4.25,0.7)--(5.25,1.85)--(6.25,1.85)--(8.25,3);
\draw (5.25,2.4)--(9.25,4.7)--(8.25,3);
\draw (5.25,2.4)--(5.25,1.85);

\filldraw[white] (30.15,0.4)--(26.75,2.4)--(26.75,1.85)--(27.75,0.7)--(29.75,-0.45)--(30.25,0.4);
\filldraw[lightgray] (26.75,1.85)--(25.75,1.85)--(23.75,3)--(22.75,4.7)--(26.75,2.4)--(26.75,1.85);
\draw [lightgray] (26.75,2.4)--(30.25,0.4);
\draw [lightgray] (27.75,0.7)--(29.75,-0.45);
\draw (27.75,0.7)--(26.75,1.85)--(25.75,1.85)--(23.75,3);
\draw (23.75,3)--(22.75,4.7);
\draw (26.75,2.4)--(22.75,4.7);
\draw (26.75,2.4)--(26.75,1.85);

%\filldraw[lightgray] (23.75,3) .. controls (21.5,4) .. (15.5,2.4)--(16.5,3.4) .. controls (21.5,5.125) ..  (23.75,4.125);
%\draw (23.75,3) .. controls (21.5,4) .. (15.5,2.4);
%\draw (23.75,4.125) .. controls (21.5,5.125) .. (16.5,3.4);
%\filldraw[lightgray] (8.25,3) .. controls (10.5,4) .. (15.5,2.4)--(16.5,3.4) .. controls (10.5,5.125) .. (8.25,4.125)--(8.25,3);
%\draw (8.25,3) .. controls (10.5,4) .. (15.5,2.4);
%\draw (8.25,4.125) .. controls (10.5,5.125) .. (16.5,3.4);

\filldraw[white] (16.5,-7)--(16.5,-3.275)--(16,-3.775)--(15.5,-5.5)--(15.5,-7)--(16.5,-7);
\filldraw[lightgray] (16,-3.775)--(15.5,-2.05)--(15.5,1.4)--(16.5,3.125)--(16.5,-3.275)--(16,-3.775);
\draw [lightgray] (15.5,-7)--(15.5,-5.5);
\draw [lightgray] (16.5,-7)--(16.5,-3.275);
\draw (15.5,-5.5)--(16,-3.775)--(15.5,-2.05)--(15.5,1.4);
\draw (16.5,-3.275)--(16.5,3.125);
\draw (16,-3.775)--(16.5,-3.275);
\draw (15.5,1.4)--(16.5,3.125);

%Multiplicities
\node at (16,10) {$2$};
\node at (22.25,10) {$1$};
\node at (9.75,10) {$1$};
\node at (16,17) {$1$};

\node at (16,4) {$3$};
\node at (14.5,-0.325) {$2$};
\node at (14.5,-3.775) {$1$};
\node at (16,-0.325) {$1$};

\node at (9.175,1.1) {$2$};
\node at (9.175,-0.9) {$1$};
\node at (7.175,0.25) {$1$};
\node at (7.5,1.95) {$2$};
\node at (5.5,1) {$1$};
\node at (11.175,-2.35) {$1$};
\node at (6.8,2.72) {$1$};

\node at (22.825,1.1) {$2$};
\node at (22.825,-0.9) {$1$};
\node at (24.825,0.25) {$1$};
\node at (24.5,1.95) {$2$};
\node at (26.5,1) {$1$};
\node at (20.825,-2.35) {$1$};
\node at (25.2,2.72) {$1$};
\end{tikzpicture}
\caption{Fibre of type $\mathrm{IV}^*$.}
\label{fig:IV*}
  \end{center}
\end{figure}

Finally, in the case $d=3$,  the threefold $\overline{\calX}$ is again non-normal along its fibre over $\delta$. This time we perform a change of coordinates $w \mapsto \delta w$; this is an isomorphism on a general fibre and performs a birational modification on the fibre over $\delta = 0$. The resulting family is given by 
\[\frac{\delta}{27}w^4+  \delta^{k+2}w^2yz+ \delta wy^2z + \delta wyz^2 +x^2yz   = 0\]
The central fibre in this family is identical to that in the family given by  Equation \eqref{eq:I0*} (see Figure \ref{fig:I0*unresolved}) and resolves in the same way; the resulting fibre of type $\mathrm{I}_0^*$ is shown in Figure \ref{fig:I0*}.
\end{proof}

\subsubsection{Types \texorpdfstring{$\mathrm{II}$, $\mathrm{IV}$, $\mathrm{I}_0$}{II, IV, I0}} Data: $\gamma = \infty$, $d \in \{1,2,3\}$, $\det(\rho) = (-1)^{d+1}$.  

These data may be realized by Equation \eqref{eq:2dimX1} when $\beta'$ has a simple pole and $\gamma$ has a pole of order $d$. However, as in the previous section, we consider the more general family $\overline{\calX} \to \Delta$ defined by the equation
\begin{equation} \label{eq:infinity2} \{\tfrac{1}{27}w^4 +  \delta^{d+k}w^2yz+ \delta^dwy^2z + \delta^dwyz^2 + \delta^{d-1}x^2yz   = 0\} \subset \bP^3[w,x,y,z] \times \Delta,\end{equation}
where $k \geq 0$ and $d \in \{1,2,3\}$. Note that if $k = 0$ this family is obtained from Equation \eqref{eq:2dimX1} by setting $\beta'=\frac{1}{\delta}$ and $\gamma = \frac{1}{\delta^d}$, then clearing denominators, but if $k > 0$ it is not in the form of Equation \eqref{eq:2dimX1}.

\begin{proposition} \label{prop:infinityfibres2} Consider the threefold $\overline{\calX} \to \Delta$ defined by Equation \eqref{eq:infinity2}. For any $k \geq 0$, after resolving the singularities of $\overline{\calX}$ we may perform a birational modification, affecting only the fibre over $\delta = 0$, so that the resulting threefold is smooth and has trivial canonical sheaf. In the cases where $d = 1$ or $2$ the resulting fibre contains $6$ or $3$ components, respectively, but does not have normal crossings. In the case $d = 3$ the resulting fibre is a smooth $M_1$-polarized K3 surface given as the minimal resolution of 
\[\{\tfrac{1}{27}w^4 + wy^2z+wyz^2+x^2yz = 0\} \subset \bP^4[w,x,y,z].\] 
The generalized homological invariant evaluated on a simple closed loop around $\delta = 0$ is conjugate to
\[\begin{pmatrix}(-1)^{d+1} & 0 & 0 \\ 0 & -\omega^d & 0 \\ 0 & 0 & -\omega^{5d} \end{pmatrix},\]
where $\omega$ is a primitive sixth root of unity. For $d = 1$, $2$, and $3$ the resulting fibre type is called $\mathrm{IV}$, $\mathrm{II}$, and $\mathrm{I}_0$ respectively.
\end{proposition}

\begin{proof} In the case $d=1$, the central fibre $\overline{\pi}^{-1}(0)$ consists of a single component which contains a curve of $A_1$ singularities, as shown in Figure \ref{fig:d=4}. The two points labelled $P$ are singularities of type $(4,0)$, but the component of the central fibre is not one of the hyperplanes from Lemma \ref{lem:singularity}. Thus, rather than using this lemma, we resolve directly by first blowing up the curve of $A_1$ singularities, which gives three new components, then blowing up again along the two curves of $A_1$ singularities that result. The resulting fibre of type $\mathrm{IV}$ has six components and is shown in Figure \ref{fig:IV}. Note that these components do not all meet normally: the bold line labelled $T$ is a line of triple points given locally as the intersection of the curves $\{s^2 + t^2 = 0\}$ (corresponding to the upper component) and $\{t = 0\}$ (corresponding to the lower component) in $\bA^2[s,t]$. 

\begin{figure}
\begin{center}
\begin{tikzpicture}[scale=0.6]
\draw [very thick] (0,0)--(8,0);
\draw (0,0)--(4,6.95)--(8,0);

\filldraw[black]
(0,0) circle (3pt)
(4,0) circle (3pt)
(8,0) circle (3pt)
(4,6.95) circle (3pt)
;

\node[above] at (5.2,-0.15) {$A_1$};
\node[right] at (7.9,0.3) {$P$};
\node[right] at (-0.7,0.3) {$P$};
\end{tikzpicture}
\caption{Partial resolution of fibre $\overline{\pi}^{-1}(0)$ in case $d = 1$ of Proposition \ref{prop:infinityfibres2}.}
\label{fig:d=4}
\end{center}
\end{figure}

\begin{figure}
\begin{center}
\begin{tikzpicture}[scale=0.6]
\draw (0,0)--(4,6.95)--(8,0);
\draw (0,0)--(0.5,-0.86)--(0.5,-1.86)--(1,-2.72)--(7,-2.72)--(7.5,-1.86)--(7.5,-0.86)--(8,0);
\draw (1,-2.72)--(2.5,-0.12)--(1.72,1.23);
\draw (0.5,-0.86)--(1.72,1.23)--(1.72,2.98);
\draw (7,-2.72)--(5.5,-0.12)--(6.28,1.23);
\draw (7.5,-0.86)--(6.28,1.23)--(6.28,2.98);
\draw [very thick] (2.5,-0.12)--(5.5,-0.12);

\node [above] at (4,-0.2) {$T$};
\node at (4,-1.5) {$1$};
\node at (1.5,-0.6) {$1$};
\node at (6.5,-0.6) {$1$};
\node at (1,0.9) {$1$};
\node at (7,0.9) {$1$};
\node at (4,3) {$1$};

\filldraw[black]
(0.5,-1.86) circle (3pt)
(4,-2.72) circle (3pt)
(7.5,-1.86) circle (3pt)
(4,6.95) circle (3pt)
;
\end{tikzpicture}
\caption{Fibre of type $\mathrm{IV}$}
\label{fig:IV}
\end{center}
\end{figure}

In the case $d = 2$, we perform a change of coordinates $w \mapsto \delta w$ and $x \mapsto \delta x$; this is an isomorphism on a general fibre and performs a birational modification on the fibre over $\delta = 0$.The resulting family is given by
\[\frac{\delta}{27}w^4 +  \delta^{k+1}w^2yz+ wy^2z + wyz^2 + x^2yz   = 0.\]
This family is singular only along the curves $C_1^{A_7}$, $C_2^{A_7}$, $C^{A_3}$, and $C^{A_1}$, and its central fibre $\delta = 0$ consists of three components $\{y = 0\}$, $\{z = 0\}$, and $\{ wy + wz + x^2 = 0\}$. This fibre is shown in Figure \ref{fig:II}. Note that the three components meet normally \emph{except} at the point $T = \{x = y = z = 0\}$, but this point is not a singular point of the threefold. Resolving the four curves $C_1^{A_7}$, $C_2^{A_7}$, $C^{A_3}$, and $C^{A_1}$, we obtain a family whose central fibre is of type $\mathrm{II}$: it contains three components, given by the transforms of the three components above.

\begin{figure}
\begin{center}
\begin{tikzpicture}[scale=0.7]
\draw(0,0)--(8,0);
\draw (0,0)--(4,6.95)--(8,0);
\draw (4,6.95)--(4,2.31);
\draw (0,0)--(4,2.31);
\draw (4,2.31)--(8,0);

\filldraw[black]
(0,0) circle (3pt)
(4,0) circle (3pt)
(8,0) circle (3pt)
(4,6.95) circle (3pt)
;

\node[below] at (3,4) {$y=0$};
\node[below] at (5,4) {$z=0$};
\node at (4,0.7) {$wy + wz + x^2 = 0$};
\node at (4,1.5) {$1$};
\node at (2.6,2.5) {$1$};
\node at (5.4,2.5) {$1$};
\node[above right] at (4,2.2) {$T$};
\end{tikzpicture}
\end{center}
\caption{Fibre of type $\mathrm{II}$.}
\label{fig:II}
\end{figure}

Finally, in the case $d=3$, we again perform a change of coordinates $w \mapsto \delta w$ and $x \mapsto \delta x$. The resulting family is given by
\[\tfrac{1}{27}w^4 +  \delta^{k+1}w^2yz  + wy^2z + wyz^2  + x^2yz = 0.\]
This family is singular only along the curves $C_1^{A_7}$, $C_2^{A_7}$, $C^{A_3}$, and $C^{A_1}$, and its central fibre $\delta = 0$ consists of a single component with equation
\[\tfrac{1}{27}w^4 + wy^2z+wyz^2+x^2yz = 0.\]
After resolving, it is a smooth K3 surface.
\end{proof}

\subsubsection{Modularity} \label{sec:modularity} Data: $\gamma = \infty$, $d \geq 1$, $\det(\rho) = \pm 1$. 

Propositions \ref{prop:infinityfibres1} and \ref{prop:infinityfibres2} give a description of all singular fibres with $\gamma = \infty$ and $d = \{1,2,3\}$. It remains to show that these are the only possibilities and to characterize when they appear.

All possible combinations of data in the $\gamma = \infty$ case are realized by either taking $\gamma = \frac{1}{\delta^d}$ in Equation \eqref{eq:X1} and clearing denominators, to get
\begin{equation}
\label{eq:inf1} \tfrac{1}{27}w^4 +  \delta^{d}w^2yz+ \delta^dwy^2z + \delta^dwyz^2 + \delta^dx^2yz   = 0,
\end{equation}
or by taking $\beta' = \frac{1}{\delta}$ and $\gamma = \frac{1}{\delta^d}$ in Equation \eqref{eq:2dimX1} and clearing denominators, to get
\begin{equation}
\label{eq:inf2} \tfrac{1}{27}w^4 +  \delta^{d}w^2yz + \delta^dwy^2z + \delta^dwyz^2 + \delta^{d-1}x^2yz   = 0.
\end{equation}

\begin{proposition} \label{prop:infinityfibresfull} Consider the threefolds $\overline{\pi}\colon\overline{\calX} \to \Delta$ defined by Equations \eqref{eq:inf1} and \eqref{eq:inf2}. After resolving the singularities of $\overline{\calX}$, we may perform a birational modification, affecting only the fibre over $\delta = 0$, so that the resulting threefold is smooth and has trivial canonical sheaf. The resulting fibres are given in the following table, where all values of $d \geq 1$ are taken modulo $6$. Descriptions of these fibres may be found in Propositions \ref{prop:infinityfibres1} and \ref{prop:infinityfibres2}.
\smallskip

\begin{center}
\begin{tabular}{|c|c|c|c|c|c|c|}
\hline
$d \pmod 6$ & $0$ & $1$ & $2$ & $3$ & $4$ & $5$ \\
\hline
Equation \eqref{eq:inf1} & $\mathrm{I}_0$ & $\mathrm{II}^*$ & $\mathrm{IV}^*$ & $\mathrm{I}_0^*$ & $\mathrm{IV}$ & $\mathrm{II}$ \\
Equation \eqref{eq:inf2} & $\mathrm{I}_0^*$ & $\mathrm{IV}$ & $\mathrm{II}$ & $\mathrm{I}_0$ & $\mathrm{II}^*$ & $\mathrm{IV}^*$ \\
\hline
\end{tabular}
\end{center}
\end{proposition}

\begin{proof} Observe that Equations \eqref{eq:inf1} and \eqref{eq:inf2} are the $k = 0$ cases of Equations \eqref{eq:infinity1} and \eqref{eq:infinity2}.  We proceed by an induction argument, based on two reduction steps. 

\emph{Step 1:} Start with a threefold given by \eqref{eq:infinity1}:
\[\tfrac{1}{27}w^4 + \delta^{d+k}w^2yz + \delta^dwy^2z + \delta^dwyz^2 + \delta^dx^2yz   = 0.\]
Perform a birational modification $w \mapsto \delta w$. The resulting family is given by
\[\tfrac{1}{27}w^4 + \delta^{d+k-2}w^2yz + \delta^{d-3}wy^2z + \delta^{d-3}wyz^2 + \delta^{d-4}x^2yz  = 0.\]
This is an equation of the form \eqref{eq:infinity2}, with $d \mapsto d-3$ and $k \mapsto k+1$.

\emph{Step 2:} Start with a threefold given by \eqref{eq:infinity2}:
\[\tfrac{1}{27}w^4 + \delta^{d+k}w^2yz + \delta^dwy^2z + \delta^dwyz^2 + \delta^{d-1}x^2yz  = 0.\]
Perform a birational modification $w \mapsto \delta w$ and $x \mapsto \delta x$. The resulting family is given by
\[\tfrac{1}{27}w^4 +  \delta^{d+k-2}w^2yz + \delta^{d-3}wy^2z + \delta^{d-3}wyz^2 + \delta^{d-3}x^2yz   = 0.\]
This is an equation of the form \eqref{eq:infinity1}, with $d \mapsto d-3$ and $k \mapsto k+1$.

\emph{Induction:} Given a family in one of the forms \eqref{eq:infinity1} or \eqref{eq:infinity2}, alternately perform reduction steps 1 and 2, above. Each step is an isomorphism on a general fibre and performs a birational modification on the fibre over $\delta = 0$, with the result of reducing $d$ by $3$ and increasing $k$ by $1$. Simply repeat until $1 \leq d \leq 3$, then conclude by applying Propositions \ref{prop:infinityfibres1} and \ref{prop:infinityfibres2}.

It remains to check the generalized homological invariants of the resulting fibres have the correct forms: this follows from Proposition \ref{prop:generalmonodromy}.
\end{proof}

\section{Minimal form and invariants} \label{sec:invariants}

\subsection{Minimal form} We now return to threefolds fibred non-isotrivially by $M_1$-polarized K3 surfaces over an arbitrary $1$-dimensional complex manifold $B$.

\begin{definition} A threefold fibred non-isotrivially by $M_1$-polarized K3 surfaces $\pi\colon \calX \to B$ is said to be in \emph{minimal form} if the singularities of $\calX$ are at worst terminal and the canonical sheaf of $\calX$ is trivial in a neighbourhood of any fibre.
\end{definition}

The results of the previous sections can be summarized by the following theorem.

\begin{theorem}\label{thm:minimalform} Every threefold $\pi\colon \calX \to B$ fibred non-isotrivially by $M_1$-polarized K3 surfaces is birational over $B$ to a threefold $\pi'\colon \calX' \to B$ in minimal form. Moreover, such a minimal form $\calX'$ is unique up to birational maps over $B$ that are isomorphisms in codimension $1$, and the singular fibres of $\calX'$ are classified by Table \ref{tab:fibres}.
\end{theorem}

\begin{proof} By Theorem \ref{thm:weierstrassform}, we may find a Weierstrass model $\overline{\pi}\colon \overline{\calX} \to B$ for $\pi\colon \calX \to B$. After performing a minimal resolution of $\overline{\pi}\colon \overline{\calX} \to B$, Theorem \ref{thm:singularfibres} says that we can perform further birational modifications, which affect only the singular fibres, to obtain a minimal form $\pi'\colon \calX' \to B$ as required. By \cite[Proposition 3.54]{bgav}, the conditions on singularities and the canonical sheaf imply that any two such minimal forms must be isomorphic in codimension $1$.
\end{proof}

The remainder of this paper will be devoted to studying the properties of such minimal forms. We begin with a sufficient condition for a minimal form to be smooth.

\begin{proposition} \label{prop:smoothness} Let $\pi\colon \calX \to B$ be a threefold fibred non-isotrivially by $M_1$-polarized K3 surfaces in minimal form. Suppose that all fibres in $\calX$ with generalized functional invariant $\gamma = -1$ have type $\mathrm{III}$ with $d = 1$. Then $\calX$ is birational to a smooth threefold $\pi'\colon \calX' \to B$ fibred non-isotrivially by $M_1$-polarized K3 surfaces in minimal form.
\end{proposition}
\begin{proof} Let $\pi\colon \calX \to B$ be a threefold satisfying the conditions of the theorem. The local analysis of the singular fibres  in Section \ref{sec:singularfibres} shows that, after possibly performing a birational modification affecting only the singular fibres, we may assume that $\calX$ is smooth apart from some isolated terminal singularities occurring over points in $B$ with $\gamma = -1$ as described in Propositions \ref{prop:III} and \ref{prop:III*}. But these singular cases are excluded by the assumptions on $\calX$.
\end{proof}

\begin{remark} The proposition above gives a sufficient condition for smoothness. The question of whether this condition is necessary is likely to be quite subtle; see Remark \ref{rem:isolatedsingularity}.
\end{remark}

\subsection{A canonical bundle formula}\label{sec:canonicalsheaf}

Our next goal is a formula for the canonical divisor. The reader may wish to compare \cite[Proposition 2.4]{cytfmqk3s} and its generalization \cite[Theorem 4.1]{cytfhrlpk3s}, which give conditions for triviality of the canonical bundle of a threefold fibred by $M_n$-polarized K3 surfaces with $n > 1$.

\begin{proposition} \label{prop:cbf} Let $\pi\colon \calX \to B$ be a  threefold fibred non-isotrivially by $M_1$-polarized K3 surfaces in minimal form over a compact curve $B$, with extended generalized functional invariant $\overline{g} \colon B \to \overline{\calM}_{M_1} \cong \bP^1$. Then the canonical divisor of $\calX$ is given by
\[K_{\calX} = \pi^*\left(K_B + \tfrac{1}{6}L + \sum_{P \in B} S_P P\right),\]
where $L$ is a divisor in the linear system $\overline{g}^*\calO(1)$ and $S_P$ is the number from the final column of Table \ref{tab:fibres} corresponding to the type of the fibre $\pi^{-1}(P)$.
\end{proposition}
\begin{proof} This follows from a minor modification of Fujino's and Mori's canonical bundle formula for K3 fibrations \cite[Theorem 1.2 and Corollary 1.3]{cbf2}. According to these results, a general formula for the canonical sheaf of a K3-fibred threefold is given by
\[K_{\calX} = \pi^*\left(K_B + \tfrac{1}{a}L + \sum_{P \in B} S_P P\right) + D.\]
The terms appearing in this formula are as follows.
\begin{itemize}
\item $L$ is a divisor in the linear system $\overline{g}^*\calO(1)$.
\item $a = kn$ is the product of the dimension $n$ of the moduli space and the weight $k$ of the automorphic forms that provide an embedding of the Baily-Borel-Satake compactification of this moduli space into a projective space. In our case $n = 1$ and $k = 6$, as ${\calM}_{M_1}$ is the classical modular curve $\bH/\mathrm{PSL}(2,\bZ)$.
\item $D$ is an effective $\bQ$-divisor which satisfies $\pi_*\calO_{\calX}(\lfloor i D \rfloor) = \calO_B$ for all $i \geq 0$. The assumption on the canonical sheaf in the definition of minimal form ensures that $D$ is trivial in our case.
\item $S_P = 1 - t_P$, where $t_P$ is the log canonical threshold of $\pi^*P$ with respect to $(\calX,-D)$ \cite[Definition 3.5]{cbf2}.
\end{itemize}

It thus remains to compute the values of the log canonical thresholds $t_P$. These may be computed on an open neighbourhood of each fibre and do not depend on the birational model of the fibre that we choose, so it suffices to compute them for the explicit fibres studied in Section \ref{sec:singularfibres}. If the support of the fibre in question is simple normal crossings then $t_P = \frac{1}{m_P}$, where $m_P$ is the multiplicity of the highest multiplicity component; this holds in all cases except $\mathrm{II}$, $\mathrm{IV}$, and when $\gamma = -1$. In the remaining cases the log canonical thresholds may be computed by proceeding to a log resolution of the pair $(\calX,\pi^*P)$; such a resolution may be achieved by at most two blow-ups.
\end{proof}

\subsection{Betti numbers} \label{sec:betti}

We conclude by computing the Betti numbers of our threefolds. Note that all of the results in this section require the threefold $\calX$ to be smooth; a sufficient condition for smoothness is given by Proposition \ref{prop:smoothness}.

Since $\calX$ is connected, we have $b_0(\calX) = 1$, and Poincar\'{e} duality gives $b_k(\calX) = b_{6-k}(\calX)$ for all $0 \leq k \leq 6$. The interesting Betti numbers are therefore $b_1(\calX)$, $b_2(\calX)$, and $b_3(\calX)$. We begin with the first Betti number.

\begin{proposition} Let $\pi\colon \calX \to B$ be a smooth threefold fibred non-isotrivially by $M_1$-polarized K3 surfaces in minimal form over a compact curve $B$. Then 
\[b_1(\calX) = 2g(B).\]
\end{proposition}
\begin{proof} We follow the method of \cite[Proposition 2.7]{cytfmqk3s}. Via the Hodge decomposition, it suffices to show that $h^1(\calX,\calO_{\calX}) = g(B)$. This will follow from the Leray spectral sequence for $\pi\colon \calX \to B$ if we can show that $R^1\pi_*\calO_{\calX} = 0$. By Proposition \ref{prop:cbf}, the canonical divisor of $\calX$ is the pull-back of a $\bQ$-Cartier $\bQ$-divisor from $B$, so Koll\'{a}r's torsion-freeness theorem \cite[Theorem 10.19]{smaf} gives that $R^1\pi_*\calO_{\calX}$ is torsion-free. But $H^1(X_p,\calO_{X_p}) = 0$ for a general fibre $X_p$ of $\calX$, so $R^1\pi_*\calO_{\calX}$ also has trivial generic fibre. We therefore obtain $R^1\pi_*\calO_{\calX} = 0$, as required.
\end{proof}

The second Betti number follows from an application of \cite[Lemma 3.2]{cytfmqk3s} (reproduced in the appendix as Lemma \ref{lem:oldb2}). The reader may wish to compare \cite[Proposition 3.5]{cytfmqk3s} and its generalization \cite[Proposition 4.3]{cytfhrlpk3s}, which compute the second Betti numbers of Calabi-Yau threefolds fibred by $M_n$-polarized K3 surfaces with $n > 1$.

\begin{proposition} \label{prop:b2} Let $\pi\colon \calX \to B$ be a smooth threefold fibred non-isotrivially by $M_1$-polarized K3 surfaces in minimal form over a compact curve $B$. Then
\[b_2(\calX) = 20 + \sum_{P \in B} (C_P - 1),\]
where $C_P$ is the number of irreducible components in the fibre $\pi^{-1}(P)$ \textup{(}which may be read off from the fourth column of Table \ref{tab:fibres}\textup{)}.
\end{proposition}
\begin{proof} This is a direct application of \cite[Lemma 3.2]{cytfmqk3s} (see Lemma \ref{lem:oldb2}).\end{proof}

The third Betti number is much more problematic; we present only a partial result. We begin by proving a mild generalization of \cite[Lemma 3.4]{cytfmqk3s} (reproduced in the appendix as Lemma \ref{lem:oldb3}), to deal with the fact that several of our singular fibres are not normal crossings.

\begin{lemma} \label{lem:b3} Let $\pi\colon \calX \to B$ be a smooth projective threefold fibred by K3 surfaces over a smooth compact curve $B$. Let $U \subset B$ denote the open set over which the fibres of $\pi$ are smooth, let $\pi^U$ denote the restriction of $\pi$ to $U$, and let $i\colon U \hookrightarrow B$ denote the inclusion. Suppose that $\calX$ admits a resolution $\widetilde{\calX} \to \calX$, so that every fibre of $\widetilde{\calX}$ is either
\begin{enumerate}
\item a K3 surface with at worst ADE singularities, or
\item a normal crossings union of smooth surfaces $S_i$ with $H^3(S_i,\bC) = 0$.
\end{enumerate}
Then $b_3(\calX) = h^1(B, i_*R^2\pi^U_*\bC)$.
\end{lemma}
\begin{proof} \cite[Lemma 3.4]{cytfmqk3s} corresponds to the statement above with $\widetilde{\calX} = \calX$; see Lemma \ref{lem:oldb3}. The proof of our result is almost identical to that one, so we skim most of the details and highlight only the important differences. 

The Leray spectral sequence gives a decomposition
\[b_3(\calX) = h^0(B,R^3\pi_*\bC) + h^1(B, R^2\pi_*\bC) + h^2(B,R^1\pi_*\bC).\]
To compute the terms in this decomposition, take cohomology of the following short exact sequence, which exists for each $n$ by \cite[Proposition 15.12]{htdcl2cpm},
\[0 \longrightarrow \calS_n \longrightarrow R^n\pi_* \bC \longrightarrow i_*R^n\pi^U_*\bC \longrightarrow 0;\]
here $\calS_n$ is a skyscraper sheaf supported on $B \setminus U$.  Using this sequence and standard results about cohomological vanishing, our statement reduces to showing that for every point $p \in B \setminus U$, the cohomology $H^0(U_p,\calS_3)$ is trivial, where $U_p$ is the preimage in $\calX$ of a small open disc containing $p$. 

The proof that $H^0(U_p,\calS_3) = 0$ for fibres $X_p = \pi^{-1}(p)$ satisfying case (1) of the lemma follows from the Clemens contraction theorem.  We may therefore restrict our attention to fibres $X_p$ satisfying case (2).

When $\widetilde{\calX} = \calX$, as in \cite[Lemma 3.4]{cytfmqk3s}, it follows from a result of Zucker \cite[Section 15]{htdcl2cpm} that for any fibre $X_p$ satisfying case (2),
\[H^0(U_p,\calS_3) = \im\big(\phi_p \colon H^3(\calX_{\Delta_p},\partial \calX_{\Delta_p}) \to H^3(X_p)\big),\]
where $\phi_p$ is a morphism of mixed Hodge structures and $\calX_{\Delta_p}$ is the preimage in $\calX$ of a small closed disc $\Delta_p$ containing $p$. To generalize to our setting it suffices to note that, by \cite[Section 15]{htdcl2cpm} again, for any resolution $\widetilde{\calX} \to \calX$ we have an \emph{inclusion}
\[H^0(U_p,\calS_3) \subset \im\big(\widetilde{\phi}_p \colon H^3(\widetilde{\calX}_{\Delta_p},\partial \widetilde{\calX}_{\Delta_p}) \to H^3(\widetilde{X}_p)\big).\]

The proof then concludes as in \cite[Lemma 3.4]{cytfmqk3s}. Results from \cite[Section 15]{htdcl2cpm} about the mixed Hodge structure on $H^3(\widetilde{\calX}_{\Delta_p},\partial \widetilde{\calX}_{\Delta_p})$  force $\im(\widetilde{\phi}_p)$ to lie inside the third weight graded piece $\Gr^W_3(H^3(\widetilde{X}_p))$, which may be computed using the Mayer-Vietoris spectral sequence. This computation shows that the assumption $H^3(S_i,\bC) = 0$ forces $\Gr^W_3(H^3(\widetilde{X}_p))$ to be trivial. Thus $H^0(U_p,\calS_3)$ is trivial also, completing the proof.
\end{proof}

Using this lemma, with some mild assumptions on the singular fibres to ensure that the relevant cohomological vanishing occurs, we can give a simple closed expression for the third Betti number. The reader may wish to compare \cite[Proposition 3.8]{cytfmqk3s} and its generalization \cite[Proposition 4.4]{cytfhrlpk3s}, which compute the third Betti numbers of Calabi-Yau threefolds fibred by $M_n$-polarized K3 surfaces with $n > 1$.

\begin{proposition} \label{prop:b3} Let $\pi\colon \calX \to B$ be a smooth threefold fibred non-isotrivially by $M_1$-polarized K3 surfaces in minimal form  over a compact curve $B$ and assume further that $\calX$ does not contain any singular fibres of type $\mathrm{I}_0^*$ with $\gamma \neq -1$. Then
\[b_3(\calX) = 6(g(B) - 1) + \sum_{P \in B} R_P,\]
where $R_P$ is the number from the seventh column of Table \ref{tab:fibres} corresponding to the type of the fibre $\pi^{-1}(P)$.
\end{proposition} 
\begin{proof} We wish to apply Lemma \ref{lem:b3}, but in order to do this we first need to check that the assumptions of that lemma hold. This is the case for most singular fibre types: all loci where these fibres have worse than normal crossings may be resolved by at most two blow-ups, and all components in the resulting resolved fibres are rational. The only problematic fibres are those of type $\mathrm{I}_0^*$ with $\gamma \neq -1$, which contain a component $S$ that is ruled over a smooth elliptic curve (this is the gray shaded component in Figure \ref{fig:I0*}). This component has $h^3(S,\bC) = 2$, so the condition $H^3(S_i,\bC) = 0$ does not hold for it; we therefore exclude such fibres by assumption.

Now let $U \subset B$ denote the open set over which the fibres of $\pi$ are smooth $M_1$-polarized K3 surfaces, let $\pi^U\colon \calX^U \to U$ denote the restriction of $\pi$ to $U$, and let $i\colon U \hookrightarrow B$ denote the inclusion. By Lemma \ref{lem:b3} we have $b_3(\calX) = h^1(B, i_*R^2\pi^U_*\bC)$. We may decompose
\[R^2\pi^U_*\bQ = (\calN\calS(\calX^U) \oplus \calT(\calX^U))\otimes \bQ\]
into a direct sum of two irreducible $\bQ$-local systems, where $\calT(\calX^U)$ is the transcendental variation of Hodge structure associated to $\calX^U$ and $\calN\calS(\calX^U)$ can be interpreted as those classes that lie in $\NS(X_p)$ for all fibres $X_p$ of $\calX^U$. Moreover, the $M_1$-polarization  forces $\calN\calS(\calX^U)$ to be a trivial local system, which has no higher cohomology, so we obtain
\[h^1(B, i_*R^2\pi^U_*\bQ) = h^1(B, i_*\calT(\calX^U) \otimes \bQ).\]

The rank of this last cohomology group may be computed using a variation on Poincar\'{e}'s formula in classical topology, due to del Angel, M\"{u}ller-Stach, van Straten and Zuo \cite{hca1pfcy3f}. In general, let $\bV$ be a $\bQ$-local system on $U$ and let $B \setminus U = \{P_1,\ldots,P_n\}$. Let $\sigma_i$ denote the monodromy transformation on a general fibre $\bV_P$ of $\bV$ associated to a simple closed counterclockwise loop around $P_i$. Then define 
\[R(P_i) = \rank(\bV_P) - \rank(\bV_P^{\gamma_i}),\]
where $\bV_P^{\gamma_i}$ is the subspace of elements of $\bV_P$ that are fixed by $\gamma_i$. With this notation, \cite[Proposition 3.6]{hca1pfcy3f} gives
\[h^1(B, i_*\bV) = 2(g(B) - 1) \rank(\bV) - \sum_{i = 1}^n R(P_i).\]
The result follows by applying this result to the local system $\calT(\calX^U)$, noting that $\rank(\calT(\calX^U)) = 3$ and the local monodromy matrices are those appearing in the fifth column of Table \ref{tab:fibres}.
\end{proof}

\appendix

\section{Prior results on Betti numbers of threefolds fibred by K3 surfaces} \label{appendix}

For completeness, in this appendix we reproduce two results from \cite{cytfmqk3s} that compute Betti numbers of threefolds  fibred by K3 surfaces.  These results are used in the proofs of Proposition \ref{prop:b2} and Lemma \ref{lem:b3}.  

Both results hold in the following general setting. $\pi\colon \calX \to B$ is a smooth projective threefold fibred by K3 surfaces over a compact curve $B$. The fibre of $\calX$ over $p \in B$ is denoted $X_p$. Denote by $i\colon U \hookrightarrow B$ the open set over which $\calX$ is smooth and the fibres of $\pi$ are smooth K3 surfaces, along with its natural embedding into $B$, and let $\pi^U\colon \calX^U \to U$ denote the restriction of $\pi$ to $U$.

\begin{lemma} \label{lem:oldb2} \textup{\cite[Lemma 3.2]{cytfmqk3s}} For any $p \in B$, let $C_p$ denote the number of irreducible components of the fibre $X_p$. Let $\rho_h$ be the rank of the subgroup of $H^2(X_p,\bC)$, for $p \in U$, that is fixed by the action of monodromy. Then we have
\[b_2(\calX) = b_4(\calX) = 1 + \rho_h + \sum_{p \in B}(C_p - 1).\]
\end{lemma}

\begin{lemma} \label{lem:oldb3} \textup{\cite[Lemma 3.4]{cytfmqk3s}} Suppose that every fibre $X_p$ is either:
\begin{enumerate}
\item a K3 surface with at worst ADE singularities, or
\item a normal crossings union of smooth surfaces $S_i$ with $H^3(S_i,\bC) = 0$.
\end{enumerate}
Then $b_3(\calX) = h^1(B, i_*R^2\pi^U_*\bC)$.
\end{lemma}

\bibliography{publications,preprints}
\bibliographystyle{amsalpha}
\end{document}